\documentclass{article}

\usepackage{natbib}

 \usepackage[final]{neurips_2019}

\usepackage[utf8]{inputenc} 
\usepackage[T1]{fontenc}    
\usepackage{hyperref}       
\usepackage{url}            
\usepackage{booktabs}       
\usepackage{amsfonts}       
\usepackage{nicefrac}       
\usepackage{microtype}      

\usepackage{fullpage}
\usepackage{algorithm}
\usepackage{algpseudocode}
\usepackage{amsmath,amsthm,amsfonts}
\usepackage{amssymb}
\usepackage{color}
\usepackage{mathrsfs}
\usepackage{enumitem}
\usepackage{bm}
\usepackage{multirow}
\usepackage{booktabs}
\usepackage{makecell}
\usepackage{graphicx}
\usepackage{subfigure}
\usepackage{caption}
\usepackage{comment}
\usepackage{cases}
\usepackage{epsfig}
\usepackage{appendix}
\usepackage{times}
\usepackage{enumitem}
\usepackage{verbatim}
\usepackage{mathtools}
\usepackage{amsbsy}

\usepackage{tcolorbox}

\include{notations}

\numberwithin{equation}{section}

\theoremstyle{definition}
\newtheorem{thm}{Theorem}[section]

\newtheorem{coro}[thm]{Corollary}
\newtheorem{rem}{Remark}[section]

\newcommand\ee{\mathrm{e}}
\newcommand\dd{\mathrm{d}}

\newcommand{\gd}{GD}
\newcommand{\nagmc}{NAG-\texttt{C}}
\newcommand{\nagmsc}{NAG-\texttt{SC}}
\newcommand{\nagm}{NAG}

\allowdisplaybreaks

\title{Acceleration via Symplectic Discretization of High-Resolution Differential Equations}

\author{ Bin Shi\\
  University of California, Berkeley\\
  \texttt{binshi@berkeley.edu} \\
   \And
  Simon S. Du\\
  Institute for Advanced Study\\
   \texttt{ssdu@ias.edu} \\
   \AND
   Weijie J. Su \\
   University of Pennsylvania \\
   \texttt{suw@wharton.upenn.edu} \\
   \And
   Michael I. Jordan \\
  University of California, Berkeley \\
   \texttt{jordan@cs.berkeley.edu} \\
}

\newenvironment{itemize*}%
{\begin{itemize}[leftmargin=*,topsep=0pt]%
		\setlength{\itemsep}{0pt}%
		\setlength{\parskip}{0pt}}%
	{\end{itemize}}
\newenvironment{enumerate*}%
{\begin{enumerate}[leftmargin=*,topsep=0pt]%
		\setlength{\itemsep}{0pt}%
		\setlength{\parskip}{0pt}}%
	{\end{enumerate}}

\begin{document}

\maketitle

\begin{abstract}
We study first-order optimization algorithms obtained by discretizing ordinary differential equations (ODEs) corresponding to Nesterov's accelerated gradient methods (NAGs) and Polyak's heavy-ball method. We consider three discretization schemes: symplectic Euler \textbf{(S)}, explicit Euler \textbf{(E)} and implicit Euler \textbf{(I)} schemes. We show that the optimization algorithm generated by applying the symplectic scheme to a high-resolution ODE proposed by \citet{shi2018understanding} achieves the accelerated rate for minimizing both strongly convex functions and convex functions. On the other hand, the resulting algorithm either fails to achieve acceleration or is impractical when the scheme is implicit, the ODE is low-resolution, or the scheme is explicit. 
\end{abstract}

\section{Introduction}
\label{sec: introduction}

In this paper, we consider unconstrained minimization problems:
\begin{align}\label{eqn:opt_prob}
\min_{x \in \mathbb{R}^n} ~ f(x),
\end{align}
where $f$ is a smooth convex function. The touchstone method in this setting is gradient descent (\gd):
\begin{align}
	x_{k+1} = x_k - s \nabla f(x_k), \label{eqn:gd}
\end{align}
where $x_0$ is a given initial point and $s>0$ is the step size. Whether there exist methods that improve on \gd\ while remaining within the framework of first-order optimization is a subtle and important question.


Modern attempts to address this question date to ~\cite{polyak1964some,polyakintroduction}, who incorporated a momentum term into the gradient step, yielding a method that is referred to as the \emph{heavy-ball method}:
\begin{align}
y_{k+1} = x_k - s\nabla f(x_k), \quad 
x_{k+1} = y_{k+1} - \alpha (x_{k} -x_{k-1}),  \label{eqn:heavy_ball}
\end{align}
where $\alpha >0$ is a momentum coefficient.
While the heavy-ball method provably attains a faster rate of \emph{local} convergence than \gd~near a minimum of $f$, it generally does not provide a guarantee of acceleration \emph{globally}~\citep{polyak1964some}. 

The next major development in first-order methods is due to Nesterov, who introduced first-order gradient methods that have a faster \emph{global} convergence rate than \gd~\citep{nesterov1983method,nesterov2013introductory}.
 For a $\mu$-strongly convex objective $f$ with $L$-Lipschitz gradients, Nesterov's \emph{accelerated gradient method} (\nagmsc) involves the following pair of update equations:
 \begin{align}
y_{k+1} = x_{k} -s \nabla f(x_k), \quad 
x_{k+1} = y_{k+1} + \frac{1-\sqrt{\mu s}}{1+\sqrt{\mu s}}\left(y_{k+1}-y_k\right). \label{eqn:nagmsc}
 \end{align}
 If one sets $s = 1/L$, then \nagmsc~enjoys a $O\left((1-\sqrt{\mu/L})^k\right)$ convergence rate, improving on the $O\left(\left(1-\mu/L\right)^k\right)$ convergence rate of \gd. Nesterov also developed an accelerated algorithm (\nagmc) targeting smooth convex functions that are not strongly convex:
\begin{align}
	y_{k+1} = x_k - s\nabla f(x_k), \quad 
	x_{k+1} = y_{k+1} +\frac{k}{k+3} (y_{k+1}-y_k).  \label{eqn:nagc}
\end{align}
This algorithm has a $O(L/k^2)$ convergence rate, which is faster than \gd's $O(L/k)$ rate.

While yielding optimal and effective algorithms, the design principle of Nesterov's accelerated gradient algorithms (\nagm) is not transparent. Convergence proofs for \nagm~often use the \emph{estimate sequence} technique, which is inductive in nature and relies on series of algebraic tricks~\citep{bubeck2015convex}.
In recent years progress has been made in the understanding of acceleration by moving to a \emph{continuous-time} formulation. In particular, \cite{su2016differential} showed that as $s \rightarrow 0$, \nagmc~converges to an ordinary differential equation (ODE) (Equation~\eqref{eqn: low_NAG-C}); moreover, 
for this ODE, \cite{su2016differential} derived a (continuous-time) convergence rate using a Lyapunov function, and further transformed this Lyapunov function to a discrete version and thereby provided a new proof of the fact that \nagmc~enjoys a $O(L/k^2)$ rate.

Further progress in this vein has involved taking a variational point of view that derives ODEs from an underlying Lagrangian rather than from a limiting argument~\citep{wibisono2016variational}.
While this approach captures many of the variations of Nesterov acceleration presented in the literature, it does not distinguish between the heavy-ball dynamics and the NAG dynamics, and thus fails to distinguish between local and global acceleration.  More recently, \cite{shi2018understanding} have returned to limiting arguments with a more sophisticated methodology.  They have derived \emph{high-resolution} ODEs for the heavy-ball method (Equation~\eqref{eqn: high_hb}), \nagmsc~(Equation~\eqref{eqn: high_NAG-SC}) and \nagmc~(Equation~\eqref{eqn: high_NAG-C}). Notably, the high-resolution ODEs for the heavy-ball dynamics and the accelerated dynamics are different. \cite{shi2018understanding} also presented Lyapunov functions for these ODEs as well as the corresponding algorithms, and showed that these Lyapunov functions can be used to derive the accelerated rates of \nagmsc ~and \nagmc. A number of other papers have also contributed to the understanding of acceleration by working in a continuous-time formulation~\citep{krichene2017acceleration, krichene2015accelerated, diakonikolas2017approximate, ghadimi2016accelerated, diakonikolas2017approximate}.

This emerging literature has thus provided a new level of understanding of design principles for accelerated optimization. The design involves an interplay between continuous-time and discrete-time dynamics.  ODEs are obtained either variationally or via a limiting scheme, and various properties of the ODEs are studied, including their convergence rate, topological aspects of their flow and their behavior under perturbation. Lyapunov functions play a key role in such analyses, and also allow aspects of the continuous-time analysis to be transferred to discrete time~\citep[see, e.g.,][]{wilson2016lyapunov}.

And yet the literature has not yet provided a full exploration of the transition from continuous-time ODEs to discrete-time algorithms.  Indeed, this transition is a non-trivial one, as evidenced by the decades of research on numerical methods for the discretization of ODEs, including most notably the sophisticated arsenal of techniques referred to as ``geometric numerical integration'' that are used for ODEs obtained from underlying variational principles~\citep{hairer2006geometric}. Recent work has begun to explore these issues; examples include the use of symplectic integrators by~\cite{betancourt2018symplectic} and the use of Runge-Kutta integration by ~\cite{zhang2018direct}.  However, these methods do not always yield proofs that accelerated rates are retained in discrete time, and when they do they involve implicit discretization, which is generally not practical except in the setting of quadratic objectives.

Thus we wish to address the following fundamental question:
\begin{center}
\emph{Can we systematically and provably obtain new accelerated methods via the numerical discretization of ordinary differential equations?  
		}
\end{center}

Our approach to this question is a dynamical systems framework based on Lyapunov theory.  Our main results are as follows:
\begin{enumerate}[leftmargin=*,topsep=0pt]%
\item 
In Section~\ref{subsec: high_nag_sc}, we consider three simple numerical discretization schemes---symplectic  Euler \textbf{(S)},  explicit Euler \textbf{(E)} and implicit Euler \textbf{(I)} schemes---to discretize the high-resolution ODE of Nesterov's accelerated method for strongly convex functions.  We show that the optimization method generated by symplectic discretization achieves a $O( (1 - O(1)\sqrt{\mu/L} )^k )$ rate, thereby attaining acceleration. In sharp contrast, the implicit scheme is not practical for implementation, and the explicit scheme, while being simple, fails to achieve acceleration.

\item In Section~\ref{subsec: high_hb}, we apply these discretization schemes to the ODE for modeling the heavy-ball method, which can be viewed as a low-resolution ODE that lacks a gradient-correction term \citep{shi2018understanding}. In contrast to the previous two cases of high-resolution ODEs, the symplectic scheme does not achieve acceleration for this low-resolution ODE. More broadly, in Appendix~\ref{sec: low_res} we present more examples of low-resolution ODEs where symplectic discretization does \textit{not} lead to acceleration.

\item Next, we apply the three simple Euler schemes to the high-resolution ODE of Nesterov's accelerated method for convex functions. Again, our Lyapunov analysis sheds light on the superiority of the symplectic scheme over the other two schemes. This is the subject of Section~\ref{sec:high-resolution-ode}.


\end{enumerate}

Taken together, the three findings have the implication that \textit{high-resolution} ODEs and \textit{symplectic} schemes are critical to achieving acceleration using numerical discretization. More precisely, in addition to allowing relatively simple implementations, symplectic schemes allow for a large step size without a loss of stability, in a manner akin to (but better than) implicit schemes. In stark contrast, in the setting of low-resolution ODEs, only the implicit schemes remain stable with a large step size, due to the lack of gradient correction. Moreover, the choice of Lyapunov function is equally essential to obtaining sharp convergence rates. This important fact is highlighted in Theorem~\ref{thm: grad_descent_convex3} in the Appendix, where we analyze \gd~by considering it as a discretization method for gradient
flow (the ODE counterpart of \gd). Using the discrete version of the Lyapunov function proposed in \cite{su2016differential} instead of the classical one, we show that \gd~in fact minimizes the squared gradient norm (choosing the best iterate so far) at a rate of $O(L^2/k^2)$. Although this rate of convergence in the problem of squared gradient norm minimization is known in the literature~\citep{nesterov2012make}, the Lyapunov function argument provides a systematic approach to obtaining this rate in this problem and others. In particular, this example demonstrates the usefulness and flexibility of Lyapunov functions as a mathematical tool for optimization problems.

\section{Preliminaries}
\label{sec:pre}
In this section, we introduce necessary notation, and review ODEs derived in previous work and three classical numerical discretization schemes.


We mostly follow the notation of \cite{nesterov2013introductory}, with slight 
modifications tailored to the present paper. Let $\mathcal{F}_{L}^1( \mathbb{R}^{n})$ 
be the class of $L$-smooth convex functions defined on $\mathbb{R}^n$; that is, 
$f \in \mathcal{F}^1_L(\mathbb{R}^n)$ if $f(y) \geq f(x) + \left\langle \nabla f(x), 
y - x \right\rangle$ for all $x, y \in \mathbb R^n$ and its gradient is 
$L$-Lipschitz continuous in the sense that 
\[
\left\| \nabla f(x) - \nabla f(y) \right\| \leq L \left\| x - y \right\|,
\]
where $\|\cdot\|$ denotes the standard Euclidean norm and $L > 0$ is the 
Lipschitz constant. The function class 
$\mathcal{F}_{L}^2(\mathbb{R}^{n})$ is the subclass of 
$\mathcal{F}_{L}^1(\mathbb{R}^{n})$ such that each $f$ has a 
Lipschitz-continuous Hessian.  For $p = 1, 2$, let 
$\mathcal{S}_{\mu,L}^p(\mathbb{R}^{n})$ denote the subclass of 
$\mathcal{F}_{L}^p(\mathbb{R}^{n})$ such that each member $f$ is 
$\mu$-strongly convex for some $0 < \mu \le L$. That is, 
$f \in  \mathcal{S}_{\mu, L}^p(\mathbb{R}^{n})$ if 
$f \in  \mathcal{F}_{L}^p(\mathbb{R}^{n})$ and $f(y) \geq f(x) + \left\langle \nabla f(x), y - x \right\rangle + \frac{\mu}{2} \left\| y - x \right\|^{2}$ for all $x, y \in \mathbb{R}^{n}$. Let $x^\star$ denote a minimizer of $f(x)$.


\subsection{Approximating ODEs}
In this section we list all of the ODEs that we will discretize in this paper.
We refer readers to recent papers by \cite{su2016differential, wibisono2016variational} and \cite{shi2018understanding} for the rigorous derivations of these ODEs. We begin with the simplest. Taking the step size $s \rightarrow 0$ in Equation~\eqref{eqn:gd}, we obtain the following ODE (gradient flow):
\begin{equation}
\label{eqn: grad_flow}
\dot{X} = - \nabla f(X),
\end{equation}
with any initial $X(0) = x_0 \in \mathbb{R}^n$. 

Next, by taking $s \rightarrow 0$ in Equation~\eqref{eqn:nagc}, \cite{su2016differential} derived the low-resolution ODE of NAG-\texttt{C}:
\begin{equation}
\label{eqn: low_NAG-C}
\ddot{X} + \frac{3}{t}\dot{X} +  \nabla f(X) = 0,
\end{equation}                                                
with $X(0) = x_{0}$ and $\dot{X}(0) = 0$. For strongly convex functions, by taking $s \rightarrow 0$, one can derive the following low-resolution ODE (see, for example, \cite{wibisono2016variational})
\begin{equation}
\label{eqn: low_hb_NAG-SC}
\ddot{X} + 2 \sqrt{\mu}\dot{X} +  \nabla f(X) = 0
\end{equation}                                      
that models both the heavy-ball method and \nagmsc. This ODE has the same initial conditions as~\eqref{eqn: low_NAG-C}.


Recently, \cite{shi2018understanding} proposed high-resolution ODEs for modeling acceleration methods. The key ingredient in these ODEs is that the $O(\sqrt{s})$ terms are preserved in the ODEs. As a result, the heavy-ball method and \nagmsc~have different models as ODEs.
	\begin{enumerate}[label = \textbf{(\alph*)},leftmargin=*,topsep=0pt]
		\item If $f \in \mathcal{S}_{\mu, L}^{1}(\mathbb{R}^n)$, the high-resolution ODE of the heavy-ball method \eqref{eqn:heavy_ball} is 
		\begin{equation}
		\label{eqn: high_hb}
		\ddot{X} + 2 \sqrt{\mu}\dot{X}  + (1 + \sqrt{\mu s})\nabla f(X) = 0,
		\end{equation}                                            
		with $X(0) = x_0$ and $\dot{X}(0) = - \frac{2 \sqrt{s} \nabla f(x_0)}{1 + \sqrt{\mu s}}$. This ODE has essentially the same properties as its low-resolution counterpart \eqref{eqn: low_hb_NAG-SC} due to the absence of $\nabla^2 f(X) \dot X$.
		\item If $f \in \mathcal{S}_{\mu, L}^{2}(\mathbb{R}^n)$, the high-resolution ODE of NAG-\texttt{SC} \eqref{eqn:nagmsc} is 
		\begin{equation}
		\label{eqn: high_NAG-SC}
		\ddot{X} + 2 \sqrt{\mu}\dot{X} +  \sqrt{s}\nabla^{2}f(X) \dot{X} + (1 + \sqrt{\mu s})\nabla f(X) = 0,
		\end{equation}                                                       
		with $X(0) = x_0$ and $\dot{X}(0) = - \frac{2 \sqrt{s} \nabla f(x_0)}{1 + \sqrt{\mu s}}$.                                  
		\item If $f \in \mathcal{F}_{L}^{2}(\mathbb{R}^n)$, the high-resolution ODE of NAG-\texttt{C} \eqref{eqn:nagc} is
		\begin{equation}
		\label{eqn: high_NAG-C}
		\ddot{X} + \frac{3}{t}\dot{X} +\sqrt{s}\nabla^{2}f(X)\dot{X}  + \left( 1 + \frac{3\sqrt{s}}{2t} \right) \nabla f(X) = 0
		\end{equation}     
		for $t \geq 3\sqrt{s}/2$, with $X(3\sqrt{s}/2) = x_0$ and $\dot{X}(3\sqrt{s}/2) = - \sqrt{s} \nabla f(x_0)$.                                                                 
	\end{enumerate}         

\subsection{Discretization schemes}


To discretize ODEs~\eqref{eqn: grad_flow}-\eqref{eqn: high_NAG-C}, we replace $\dot{X} $ by $x_{k+1} - x_{k}$, $\dot{V}$ by $v_{k+1}-v_k$ and replace other terms with approximations.
Different discretization schemes correspond to different approximations.
\begin{itemize}[leftmargin=*,topsep=0pt]
	
	\item The most straightforward scheme is the explicit scheme, which uses the following approximation rule:
	\[
	x_{k + 1} - x_{k} = \sqrt{s}v_{k}, \qquad \sqrt{s} \nabla^{2} f(x_{k})v_{k} \approx \nabla f(x_{k + 1}) - \nabla f(x_{k}).
	\]
	
	\item  Another discretization scheme is the implicit scheme, which uses the following approximation rule:
	\[
	x_{k + 1} - x_{k} = \sqrt{s}v_{k + 1}, \qquad \sqrt{s} \nabla^{2} f(x_{k + 1})v_{k + 1} \approx \nabla f(x_{k + 1}) - \nabla f(x_{k }).
	\]
	Note that compared with the explicit scheme, the implicit scheme is not practical because the update of $x_{k+1}$ requires knowing $v_{k+1}$ while the update of $v_{k+1}$ requires knowing $x_{k+1}$. 
	\item The last discretization scheme considered in this paper is the symplectic scheme, which uses the following approximation rule.
	\[
	x_{k + 1} - x_{k} = \sqrt{s}v_{k}, \qquad \sqrt{s} \nabla^{2} f(x_{k + 1})v_{k} \approx \nabla f(x_{k + 1}) - \nabla f(x_{k}).
	\]
Note this scheme is practical because the update of $x_{k+1}$ only requires knowing $v_{k}$.
\end{itemize}
We remark that for low-resolution ODEs, there is no  $ \nabla^{2} f(x)$ term, whereas for high-resolution ODEs, we have this term and we use the difference of gradients to approximate this term. This additional approximation term is critical to acceleration.




\section{High-Resolution ODEs for Strongly Convex Functions}
\label{sec: high_res}

This section considers numerical discretization of the high-resolution ODEs of NAG-\texttt{SC} and the heavy-ball method using the symplectic Euler, explicit Euler and implicit Euler scheme. In particular, we compare rates of convergence towards the objective minimum of the three simple Euler schemes and the two methods (NAG-\texttt{SC} and the heavy-ball method) in Section~\ref{subsec: high_nag_sc} and Section~\ref{subsec: high_hb}, respectively. For both cases, the associated symplectic scheme is shown to exhibit surprisingly similarity to the corresponding classical method.


\subsection{NAG-\texttt{SC}}
\label{subsec: high_nag_sc}

The high-resolution ODE~\eqref{eqn: high_NAG-SC} of NAG-\texttt{SC} can be equivalently written in the phase space as
\begin{equation}
\label{eqn: low_nag-sc_phase}
           \dot{X} = V,\qquad                                             
           \dot{V} = - 2 \sqrt{\mu} V - \sqrt{s} \nabla^{2} f(X) V - (1 + \sqrt{\mu s})\nabla f(X), 
\end{equation}
with the initial conditions $X(0) = x_0$ and $V(0) = - \frac{2\sqrt{s} \nabla f(x_0)}{1 + \sqrt{\mu s}}$. For any $f \in \mathcal{S}_{\mu, L}^2(\mathbb{R}^n)$, Theorem 1 of \cite{shi2018understanding} shows that the solution $X = X(t)$ of the ODE~\eqref{eqn: high_NAG-SC} satisfies
\[
f(X) - f(x^\star) \leq \frac{  2\left\| x_0 - x^\star \right\|^2}{s} \ee^{- \frac{\sqrt{\mu} t}{4}},
\]
for any step size $0 < s \le 1/L$. In particular, setting the step size to $s = 1/L$, we get
\[
f(X) - f(x^\star) \leq 2L \left\| x_0 - x^\star \right\|^2 \ee^{- \frac{\sqrt{\mu} t}{4}}.                                       
\]
In the phase space representation, NAG-\texttt{SC} is formulated as
\begin{equation}\label{eq:nag-sc-phase}
\left\{ \begin{aligned}
& x_{k+1} - x_{k} =  \sqrt{s} v_{k} \\
& v_{k+1} - v_{k} = - \frac{2\sqrt{\mu s}}{1 - \sqrt{\mu s}} 
v_{k+1} - \sqrt{s}( \nabla f(x_{k+1}) - \nabla f(x_{k}) ) - 
 \frac{1 + \sqrt{\mu s}}{1 - \sqrt{\mu s}} \cdot \sqrt{s} \nabla f( x_{k+1} ),
         \end{aligned} \right. 
\end{equation}
with the initial condition $v_{0} = - \frac{2\sqrt{s} \nabla f(x_0)}{1 + \sqrt{\mu s}}$ for any $x_0$. This method maintains the accelerated rate of the ODE by recognizing
\[
f(x_{k}) - f(x^\star) \leq \frac{5L \left\| x_0 - x^\star \right\|^{2}}{ ( 1 + \sqrt{\mu/L}/12 )^{k} };
\]
(see Theorem 3 in \cite{shi2018understanding}) and the identification $t \approx k \sqrt{s}$.



Viewing NAG-\texttt{SC} as a numerical discretization of~\eqref{eqn: high_NAG-SC}, one might wonder if any of the three simple Euler schemes---symplectic Euler scheme, explicit Euler scheme, and implicit Euler scheme---maintain the accelerated rate in discretizing the high-resolution ODE. For clarity, the update rules of the three schemes are given as follows, each with the initial points $x_{0}$ and $v_{0} = - \frac{2\sqrt{s} \nabla f(x_0)}{1 + \sqrt{\mu s}}$. 


\textbf{Euler scheme of~\eqref{eqn: low_nag-sc_phase}: (S), (E) and (I) respectively}
\begin{align*}
& \textbf{(S)}  
&& \left\{ \begin{aligned}
          & x_{k + 1} - x_{k} = \sqrt{s} v_{k}                                                          \\
          & v_{k + 1} - v_{k} = - 2 \sqrt{\mu s} v_{k + 1}- \sqrt{s} \left( \nabla f(x_{k + 1}) - \nabla f(x_{k}) \right)- \sqrt{s}(1 + \sqrt{\mu s}) \nabla f(x_{k + 1}).
         \end{aligned} \right.  \\
& \textbf{(E)}
&&\left\{ \begin{aligned}
          & x_{k + 1} - x_{k} = \sqrt{s} v_{k}                                                          \\
          & v_{k + 1} - v_{k} = - 2 \sqrt{\mu s} v_{k}- \sqrt{s} \left( \nabla f(x_{k + 1}) - \nabla f(x_{k}) \right)- \sqrt{s}(1 + \sqrt{\mu s}) \nabla f(x_{k}).  
         \end{aligned} \right. \\
&\textbf{(I)}
&& \left\{ \begin{aligned}
          & x_{k + 1} - x_{k} = \sqrt{s} v_{k + 1}                                                          \\
          & v_{k + 1} - v_{k} = - 2 \sqrt{\mu s} v_{k + 1}- \sqrt{s} \left( \nabla f(x_{k + 1}) - \nabla f(x_{k}) \right)- \sqrt{s}(1 + \sqrt{\mu s}) \nabla f(x_{k + 1}).
         \end{aligned} \right. 
\end{align*}

Among the three Euler schemes, the symplectic scheme is the \textit{closest} to NAG-\texttt{SC}~\eqref{eq:nag-sc-phase}. More precisely, NAG-\texttt{SC} differs from the symplectic scheme only in an additional factor of $\frac1{1 - \sqrt{\mu s}}$ in the second line of \eqref{eq:nag-sc-phase}. When the step size $s$ is small, NAG-\texttt{SC} is, roughly speaking, a symplectic method if we make use of $\frac1{1 - \sqrt{\mu s}} \approx 1$. In relating to the literature, the connection between accelerated methods and the symplectic schemes has been explored in \cite{betancourt2018symplectic}, which mainly considers the leapfrog integrator, a second-order symplectic integrator. In contrast, the symplectic Euler scheme studied in this paper is a first-order symplectic integrator.


Interestingly, the close resemblance between the two algorithms is found not only in their formulations, but also in their convergence rates, which are \textit{both} accelerated as shown by Theorem~\ref{thm: high_nag-sc-three} and Theorem~\ref{thm:fixed-s-nag-three}.

Note that the discrete Lyapunov function used in the proof of the symplectic Euler scheme of~\eqref{eqn: low_nag-sc_phase} is
\begin{align}
\label{eqn: Lypunov_symplectic_sc}
\mathcal{E}(k) = &    \frac{1}{4} \left\| v_{k} \right\|^{2} + \frac{1}{4} \left\| 2\sqrt{\mu}(x_{k + 1} - x^\star) + v_{k}  + \sqrt{s} \nabla f(x_{k}) \right\|^{2} \nonumber \\
                           &   + \left( 1 + \sqrt{\mu s} \right) \left( f(x_{k}) - f(x^\star) \right) - \frac{( 1 + \sqrt{\mu s} )^{2}}{1 + 2 \sqrt{\mu s}} \cdot \frac{s}{2} \left\| \nabla f(x_{k}) \right\|^{2}.
\end{align}

The proof of Theorem~\ref{thm: high_nag-sc-three} is deferred to Appendix~\ref{subsec: high-nag-sc}. The following result is a useful consequence of this theorem.
\begin{thm}[Discretization of NAG-\texttt{SC} ODE]\label{thm:fixed-s-nag-three}
For any $f \in \mathcal{S}_{\mu, L}^{1}(\mathbb{R}^n)$, the following conclusions hold:
\begin{enumerate}
\item[(a)]
Taking step size $s = 4/(9L)$, the symplectic Euler scheme of~\eqref{eqn: low_nag-sc_phase} satisfies
\begin{align}
\label{eqn: conver_nag-sc_sym_high_fix1}
f(x_k) - f(x^\star) \leq  \frac{5L\left\| x_0 - x^\star \right\|^2}{\left( 1 + \frac{1}{9} \sqrt{ \frac{\mu}{L} } \right)^{k}}.                                       
 \end{align}                  

\item[(b)] Taking step size $s = \mu / ( 100 L^2 ) $, the explicit Euler scheme of~\eqref{eqn: low_nag-sc_phase} satisfies
  \begin{align}
\label{eqn: conver_nag-sc_ex_high_fix1}
f(x_k) - f(x^\star) \leq   3 L\left\| x_0 - x^\star \right\|^2 \left( 1 - \frac{\mu}{80L} \right)^{k}.                                       
 \end{align}   

\item[(c)] Taking step size $s = 1/L$, the implicit Euler scheme  of~\eqref{eqn: low_nag-sc_phase} satisfies
\begin{align}
\label{eqn: conver_nag-sc_im_high_fix1}
f(x_k) - f(x^\star) \leq \frac{ 13 \left\| x_0 - x^\star \right\|^2}{ 4 \left( 1 + \frac{1}{4} \sqrt{\frac{\mu}{L}} \right)^{k} }.                                       
 \end{align}     
\end{enumerate}

\end{thm}

In addition, Theorem~\ref{thm:fixed-s-nag-three} shows that the implicit scheme also achieves acceleration. However, unlike NAG-\texttt{SC}, the symplectic scheme, and the explicit scheme, the implicit scheme is generally not easy to use in practice because it requires solving a nonlinear fixed-point equation when the objective is not quadratic. On the other hand, the explicit scheme can only take a smaller step size $O(\mu/L^2)$, which prevents this scheme from achieving acceleration.

\subsection{The heavy-ball method}
\label{subsec: high_hb}


We turn to the heavy-ball method ODE \eqref{eqn: high_hb}, whose phase space representation reads
\begin{equation}
\label{eqn: low_hb_phase}
           \dot{X} = V,                                             \qquad
           \dot{V} = - 2 \sqrt{\mu} V - (1 + \sqrt{\mu s})\nabla f(X),
\end{equation}
with the initial conditions $X(0) = x_0$ and $V(0) = - \frac{2\sqrt{s} \nabla f(x_0)}{1 + \sqrt{\mu s} }$. Theorem 2 in \cite{shi2018understanding} shows that the solution $X = X(t)$ to this ODE satisfies
\[
f(X(t)) - f(x^\star) \leq \frac{ 7 \left\| x_0 - x^\star \right\|^2}{2s} \ee^{- \frac{\sqrt{\mu} t}{4}},
\]
for $f \in \mathcal{S}_{\mu, L}^{1}(\mathbb{R}^n)$ and any step size $0 < s \leq 1/L$. In particular, taking $s = 1/L$ gives
\[
f(X(t)) - f(x^\star) \leq \frac{ 7L \left\| x_0 - x^\star \right\|^2}{2} \ee^{- \frac{\sqrt{\mu} t}{4}}.
\]

Returning to the discrete regime, Polyak's heavy-ball method uses the following update rule:
\[
  \left\{ \begin{aligned}
          & x_{k + 1} - x_{k} = \sqrt{s} v_{k}                                                          \\
          & v_{k + 1} - v_{k} = - \frac{2 \sqrt{\mu s}}{1 - \sqrt{\mu s}} v_{k + 1} - \frac{1 + \sqrt{\mu s}}{1 - \sqrt{\mu s}} \cdot \sqrt{s} \nabla f(x_{k + 1}),
         \end{aligned} \right. 
\]
which attains a non-accelerated rate (see Theorem 4 of \cite{shi2018understanding}):
\begin{equation}\label{eq:heavy-ball-dis-rate}
f(x_k) - f(x^\star) \le \frac{5L \left\| x_0 - x^\star \right\|^{2}}{ \left( 1 + \frac{\mu}{16L}\right)^{k}}.
\end{equation}

The three simple Euler schemes for numerically solving the ODE~\eqref{eqn: high_hb} are given as follows. Every scheme starts with any arbitrary $x_{0}$ and $v_{0} = - \frac{2\sqrt{s} \nabla f(x_0)}{1 + \sqrt{\mu s} }$. As in the case of NAG-\texttt{SC}, the symplectic scheme is the closest to the heavy-ball method.



\textbf{Euler scheme of~\eqref{eqn: low_hb_phase}: (S), (E) and (I) respectively}
\begin{align*}
&\textbf{(S)} &&\left\{ \begin{aligned}
               & x_{k + 1} - x_{k} = \sqrt{s} v_{k},                    \\
               & v_{k + 1} - v_{k} = - 2 \sqrt{\mu s} v_{k + 1} - \sqrt{s}(1 + \sqrt{\mu s}) \nabla f(x_{k + 1}).
             \end{aligned} \right. \\
&\textbf{(E)} &&\left\{ \begin{aligned}
          & x_{k + 1} - x_{k} = \sqrt{s} v_{k}                                                                                      \\
          & v_{k + 1} - v_{k} = - 2 \sqrt{\mu s} v_{k} - \sqrt{s}(1 + \sqrt{\mu s}) \nabla f(x_{k}).
         \end{aligned} \right. \\
&\textbf{(I)} &&\left\{ \begin{aligned}
          & x_{k + 1} - x_{k} = \sqrt{s} v_{k + 1}                                                          \\
          & v_{k + 1} - v_{k} = - 2 \sqrt{\mu s} v_{k + 1} - \sqrt{s}(1 + \sqrt{\mu s}) \nabla f(x_{k + 1}).  
         \end{aligned} \right. 
\end{align*}

The theorem below characterizes the convergence rates of the three schemes. This theorem is extended to general step sizes by Theorem~\ref{thm:heavy-ball-rate-general-s} in Appendix~\ref{subsec: proof_polyak}.

\begin{thm}[Discretization of heavy-ball ODE]\label{thm:heavy-ball-rate-fixed-s}
For any $f \in \mathcal{S}_{\mu, L}^{1}(\mathbb{R}^n)$, the following conclusions hold:
\begin{enumerate}
\item[(a)] Taking step size $s = \mu / (16L^{2})$, the symplectic Euler scheme of~\eqref{eqn: low_hb_phase} satisfies
 \begin{align}
\label{eqn: conver_high_hb_sym_fix}
f(x_k) - f(x^\star) \leq \frac{3L \left\| x_0 - x^\star \right\|^{2} }{\left(1 + \frac{\mu}{16L} \right)^{k}}.                                       
 \end{align}

\item[(b)] Taking step size $s = \mu/(36L^{2})$, the explicit Euler scheme of~\eqref{eqn: low_hb_phase} satisfies
 \begin{align}
\label{eqn: conver_high_hb_ex_fix}
f(x_k) - f(x^\star) \leq 3L \left\| x_0 - x^\star \right\|^{2} \left(1 - \frac{\mu}{48L} \right)^{k}.                                  
 \end{align}    

\item[(c)] Taking step size $s = 1/L$, the implicit Euler scheme of~\eqref{eqn: low_hb_phase} satisfies
\begin{align}
\label{eqn: conver_high_hb_im_fix}
f(x_k) - f(x^\star) \leq \frac{ 15L \left\| x_0 - x^\star \right\|^{2} }{4\left(1 + \frac{1}{4} \sqrt{\frac{\mu}{L}} \right)^{k} }.                                     
 \end{align}  

\end{enumerate}
  
\end{thm}

Taken together, \eqref{eq:heavy-ball-dis-rate} and Theorem~\ref{thm:heavy-ball-rate-fixed-s} imply that neither the heavy-ball method nor the symplectic scheme attains an accelerated rate. In contrast, the implicit scheme achieves acceleration as in the NAG-\texttt{SC} case, but it is impractical except for quadratic objectives.

\section{High-Resolution ODEs for Convex Functions}
\label{sec:high-resolution-ode}


In this section, we turn to numerical discretization of the high-resolution ODE~\eqref{eqn: high_NAG-C} related to NAG-\texttt{C}. All proofs are deferred to Appendix~\ref{sec: high_nag-c}. 
This ODE in the phase space representation reads~\citep{shi2018understanding} 
as follows:
\begin{equation}
\label{eqn: low_nag-c_phase}
           \dot{X} = V,                                             \quad
           \dot{V} = - \frac{3}{t} \cdot V - \sqrt{s} \nabla^{2} f(X) V - \left( 1 + \frac{3\sqrt{s}}{2t} \right)\nabla f(X),  
\end{equation}
with $X(3\sqrt{s}/2) = x_0$ and $V(3\sqrt{s}/2) = -\sqrt{s} \nabla f(x_0)$. 
Theorem 5 of \cite{shi2018understanding} shows that
\label{thm: high_nag-c_con}
 Let $f \in \mathcal{F}_{L}^{1}(\mathbb{R}^n)$. For any step size $0 < s \leq 1/L$, the solution $X = X(t)$ of the high-resolution ODE~\eqref{eqn: high_NAG-C} satisfies 
\begin{equation}
 \label{eqn: conver_NAG-cC_high1}
\left\{ \begin{aligned}
            & f(X) - f(x^\star)  \leq \frac{\left( 4 + 3sL \right) \left\| x_0 - x^\star \right\|^{2}}{ t \left( 2t + \sqrt{s} \right) } \\
            & \inf_{t_{0} \leq u \leq t}\left\| \nabla f(X(u)) \right\|^{2} \leq  \frac{ \left( 12 + 9sL \right)\left\| x_0 - x^\star \right\|^{2} }{ 2\sqrt{s} \left( t^{3} - t_0^3 \right) }                               
 \end{aligned} \right. ,
 \end{equation}
 for any $t > t_0 = 1.5 \sqrt{s}$.                                
A caveat here is that it is unclear how to use a Lyapunov function to prove convergence of the (simple) explicit, symplectic or implicit Euler scheme by direct numerical discretization of the ODE~\eqref{eqn: low_NAG-C}. See  Appendix~\ref{subsec: standard_numerical} for more discussion on this point. Therefore, we slightly modify the ODE to the following one:
\begin{equation}
 \label{eqn: nag-c_phase}
           \dot{X} = V,                                             \quad
           \dot{V} = - \frac{3}{t} \cdot V - \sqrt{s} \nabla^{2} f(X) V - \left( 1 + \frac{3\sqrt{s}}{t} \right)\nabla f(X).  
\end{equation}
The only difference is in the third term on the right-hand side of the second equation, where we replace $\left( 1 + \frac{3\sqrt{s}}{2t} \right)\nabla f(X)$ by $\left( 1 + \frac{3\sqrt{s}}{t} \right)\nabla f(X)$. Now, we apply the three schemes on this (modified) ODE in the phase space, including the original NAG-\texttt{C}, which all start with $x_0$ and $v_0 = - \sqrt{s} \nabla f(x_0)$.

\textbf{Euler scheme of~\eqref{eqn: nag-c_phase}: (S), (E) and (I) respectively}
\begin{align*}
&\textbf{(S)} &&\left\{ \begin{aligned}
          & x_{k + 1} - x_{k} = \sqrt{s} v_{k }                                                          \\
          & v_{k + 1} - v_{k} = - \frac{3}{k + 1}v_{k + 1} - \sqrt{s} \left( \nabla f(x_{k + 1}) - \nabla f(x_{k}) \right)- \sqrt{s}\left(\frac{k + 4}{ k + 1}\right) \nabla f(x_{k + 1}).  
         \end{aligned} \right.  \\
&\textbf{(E)} &&\left\{\begin{aligned}
          & x_{k + 1} - x_{k} = \sqrt{s} v_{k}                                                          \\
          & v_{k + 1} - v_{k} = - \frac{3}{k}v_{k} - \sqrt{s} \left( \nabla f(x_{k + 1}) - \nabla f(x_{k}) \right)- \sqrt{s}\left(\frac{k + 3}{k}\right) \nabla f(x_{k}).  
         \end{aligned} \right. \\
&\textbf{(I)} &&\left\{\begin{aligned}
          & x_{k + 1} - x_{k} = \sqrt{s} v_{k + 1}                                                          \\
          & v_{k + 1} - v_{k} = - \frac{3}{k + 1}v_{k + 1} - \sqrt{s} \left( \nabla f(x_{k + 1}) - \nabla f(x_{k}) \right)- \sqrt{s}\left( \frac{k + 4}{ k + 1}\right) \nabla f(x_{k + 1}).  
         \end{aligned} 
         \right. 
%
\end{align*}

\begin{thm}\label{thm:nag-c-all-phase-space}
Let $f \in \mathcal{F}_{L}^{1} \left( \mathbb{R}^{n} \right)$. The following statements are true:
\begin{enumerate}[leftmargin=*,topsep=0pt]
\item[(a)] 
For any step size $0 < s \leq 1/(3L)$, the symplectic Euler scheme of~\eqref{eqn: nag-c_phase} (original NAG-\texttt{C}) 
satisfies
  \begin{equation}
\label{eqn: NAG-C}
          f(x_{k}) - f(x^\star) \leq \frac{119 \left\| x_{0} - x^\star \right\|^{2}}{s(k + 1)^{2}}, \quad  \min_{0 \leq i \leq k} \left\| \nabla f(x_{i}) \right\|^{2} \leq \frac{8568 \left\| x_{0} - x^\star \right\|^{2} }{s^2(k + 1)^{3}} ;
\end{equation}              
  

\item[(b)]
Taking any step size $0 < s \leq 1/L $, the implicit Euler scheme of~\eqref{eqn: nag-c_phase} 
satisfies
\begin{equation}
\label{eqn: NAG-C_im}
          f(x_{k}) - f(x^\star) \leq \frac{\left( 3sL + 2 \right) \left\| x_{0} - x^\star \right\|^{2} }{s(k + 2)(k + 3)}, \quad 
          \min_{0 \leq i \leq k} \left\| \nabla f(x_{i}) \right\|^{2} \leq \frac{ \left( 3sL + 2 \right) \left\| x_{0} - x^\star \right\|^{2} }{s^{2} (k + 1 )^3}.
\end{equation}        

\end{enumerate}
  
\end{thm}

Note that Theorem~\ref{thm:nag-c-all-phase-space} (a) is the same as Theorem 6 of \cite{shi2018understanding}. The explicit Euler scheme does not guarantee convergence; see the analysis in Appendix~\ref{subsec: phase-space_numercial}.

\section{Discussion}
\label{sec: conclusion}

In this paper, we have analyzed the convergence rates of three numerical discretization schemes---the symplectic Euler scheme, explicit Euler scheme, and implicit Euler scheme---applied to ODEs that are used for modeling Nesterov's accelerated methods and Polyak's heavy-ball method. The symplectic scheme is shown to achieve accelerated rates for the high-resolution ODEs of NAG-\texttt{SC} and (slightly modified) NAG-\texttt{C} \citep{shi2018understanding}, whereas no acceleration rates are observed when the same scheme is used to discretize the low-resolution counterparts \citep{su2016differential}. For comparison, the explicit scheme only allows for a small step size in discretizing these ODEs in order to ensure stability, thereby failing to achieve acceleration. Although the implicit scheme is proved to yield accelerated methods no matter whether high-resolution or low-resolution ODEs are discretized, this scheme is generally not practical except for a limited number of cases (for example, quadratic objectives).  

We conclude this paper by presenting several directions for future work. This work suggests that both symplectic schemes and high-resolution ODEs are crucial for numerical discretization to achieve acceleration. It would be of interest to formalize and prove this assertion. For example, does any higher-order symplectic scheme maintain acceleration for the high-resolution ODEs of NAGs? What is the fundamental mechanism of the gradient correction in high-resolution ODE in stabilizing symplectic discretization? Moreover, since the discretizations are applied to the modified high-resolution ODE of NAG-\texttt{C}, it is tempting to perform a comparison study between the two high-resolution ODEs in terms of discretization properties. Finally, recognizing Nesterov's method
(NAG-\texttt{SC}) is very similar to, but still different from, the corresponding symplectic scheme, one can design new algorithms as interpolations of the two methods; it would be interesting to investigate the convergence properties of these new algorithms.

\bibliographystyle{plainnat}
\bibliography{sigproc}

\begin{thebibliography}{17}
\providecommand{\natexlab}[1]{#1}
\providecommand{\url}[1]{\texttt{#1}}
\expandafter\ifx\csname urlstyle\endcsname\relax
  \providecommand{\doi}[1]{doi: #1}\else
  \providecommand{\doi}{doi: \begingroup \urlstyle{rm}\Url}\fi

\bibitem[Betancourt et~al.(2018)Betancourt, Jordan, and
  Wilson]{betancourt2018symplectic}
Michael Betancourt, Michael~I Jordan, and Ashia~C Wilson.
\newblock On symplectic optimization.
\newblock \emph{arXiv preprint arXiv:1802.03653}, 2018.

\bibitem[Bubeck(2015)]{bubeck2015convex}
S{\'e}bastien Bubeck.
\newblock Convex optimization: Algorithms and complexity.
\newblock \emph{Foundations and Trends in Machine Learning}, 8\penalty0
  (3-4):\penalty0 231--357, 2015.

\bibitem[Diakonikolas and Orecchia(2017)]{diakonikolas2017approximate}
Jelena Diakonikolas and Lorenzo Orecchia.
\newblock The approximate duality gap technique: A unified theory of
  first-order methods.
\newblock \emph{arXiv preprint arXiv:1712.02485}, 2017.

\bibitem[Ghadimi and Lan(2016)]{ghadimi2016accelerated}
Saeed Ghadimi and Guanghui Lan.
\newblock Accelerated gradient methods for nonconvex nonlinear and stochastic
  programming.
\newblock \emph{Mathematical Programming}, 156\penalty0 (1-2):\penalty0 59--99,
  2016.

\bibitem[Hairer et~al.(2006)Hairer, Lubich, and Wanner]{hairer2006geometric}
Ernst Hairer, Christian Lubich, and Gerhard Wanner.
\newblock \emph{Geometric numerical integration: structure-preserving
  algorithms for ordinary differential equations}, volume~31.
\newblock Springer Science \& Business Media, 2006.

\bibitem[Krichene and Bartlett(2017)]{krichene2017acceleration}
Walid Krichene and Peter~L Bartlett.
\newblock Acceleration and averaging in stochastic descent dynamics.
\newblock In \emph{Advances in Neural Information Processing Systems}, pages
  6796--6806, 2017.

\bibitem[Krichene et~al.(2015)Krichene, Bayen, and
  Bartlett]{krichene2015accelerated}
Walid Krichene, Alexandre Bayen, and Peter~L Bartlett.
\newblock Accelerated mirror descent in continuous and discrete time.
\newblock In \emph{Advances in Neural Information Processing Systems}, pages
  2845--2853, 2015.

\bibitem[Nesterov(1983)]{nesterov1983method}
Yurii Nesterov.
\newblock A method of solving a convex programming problem with convergence
  rate {$O(1/k^2)$}.
\newblock \emph{Soviet Mathematics Doklady}, 27\penalty0 (2):\penalty0
  372--376, 1983.

\bibitem[Nesterov(2012)]{nesterov2012make}
Yurii Nesterov.
\newblock How to make the gradients small.
\newblock \emph{Optima}, 88:\penalty0 10--11, 2012.

\bibitem[Nesterov(2013)]{nesterov2013introductory}
Yurii Nesterov.
\newblock \emph{Introductory Lectures on Convex Optimization: A Basic Course},
  volume~87.
\newblock Springer Science \& Business Media, 2013.

\bibitem[Polyak(1964)]{polyak1964some}
Boris~T Polyak.
\newblock Some methods of speeding up the convergence of iteration methods.
\newblock \emph{USSR Computational Mathematics and Mathematical Physics},
  4\penalty0 (5):\penalty0 1--17, 1964.

\bibitem[Polyak(1987)]{polyakintroduction}
Boris~T Polyak.
\newblock Introduction to optimization.
\newblock \emph{Optimization Software, Inc, New York}, 1987.

\bibitem[Shi et~al.(2018)Shi, Du, Jordan, and Su]{shi2018understanding}
Bin Shi, Simon~S Du, Michael~I Jordan, and Weijie~J Su.
\newblock Understanding the acceleration phenomenon via high-resolution
  differential equations.
\newblock \emph{arXiv preprint arXiv:1810.08907}, 2018.

\bibitem[Su et~al.(2016)Su, Boyd, and Cand\`es]{su2016differential}
Weijie Su, Stephen Boyd, and Emmanuel~J Cand\`es.
\newblock A differential equation for modeling {N}esterov's accelerated
  gradient method: theory and insights.
\newblock \emph{Journal of Machine Learning Research}, 17\penalty0
  (153):\penalty0 1--43, 2016.

\bibitem[Wibisono et~al.(2016)Wibisono, Wilson, and
  Jordan]{wibisono2016variational}
Andre Wibisono, Ashia~C Wilson, and Michael~I Jordan.
\newblock A variational perspective on accelerated methods in optimization.
\newblock \emph{Proceedings of the National Academy of Sciences}, 113\penalty0
  (47):\penalty0 E7351--E7358, 2016.

\bibitem[Wilson et~al.(2016)Wilson, Recht, and Jordan]{wilson2016lyapunov}
Ashia~C Wilson, Benjamin Recht, and Michael~I Jordan.
\newblock A {L}yapunov analysis of momentum methods in optimization.
\newblock \emph{arXiv preprint arXiv:1611.02635}, 2016.

\bibitem[Zhang et~al.(2018)Zhang, Mokhtari, Sra, and
  Jadbabaie]{zhang2018direct}
Jingzhao Zhang, Aryan Mokhtari, Suvrit Sra, and Ali Jadbabaie.
\newblock Direct {R}unge--{K}utta discretization achieves acceleration.
\newblock \emph{arXiv preprint arXiv:1805.00521}, 2018.

\end{thebibliography}

 \newpage
\appendix
\tableofcontents

\section{Gradient Flow}
\label{sec: gradient_flow}
\subsection{Convergence rate of gradient flow}
\label{appendix: grad_flow}

The following theorem is the continuous-time version of Theorem 2.1.15 in~\cite{nesterov2013introductory}.
\begin{thm}
\label{thm: grad_flow_strongly}
Let $f \in \mathcal{S}_{\mu, L}^{1}(\mathbb{R}^n)$. The solution $X = X(t)$ to the gradient flow~\eqref{eqn: grad_flow} satisfies
\begin{equation*}
\label{eqn: gradient_flow_strongly}
\left\| X - x^{\star} \right\| \leq \ee^{- \mu t} \left\| x_0 - x^{\star} \right\|.
\end{equation*}
\end{thm}
\begin{proof}
Taking the following Lyapunov function 
\[
\mathcal{E} = \left\| X - x^{\star} \right\|^{2},
\]
we calculate its time derivative as
\begin{align*}
\frac{\dd \mathcal{E}}{\dd t}  & =         2 \left\langle \dot{X},  X - x^{\star} \right\rangle          \\
                                                                        & =       - 2 \left\langle \nabla f(X),  X - x^{\star} \right\rangle    \\
                                                                        & \leq   - 2 \mu   \left\| X - x^{\star} \right\|^{2}.
\end{align*}
Thus, we complete the proof.
\end{proof}


The theorem below is a continuous version of Theorem 2.1.14 in~\cite{nesterov2013introductory}.
\begin{thm}
\label{thm: grad_flow_convex1}
Let $f \in \mathcal{F}_{L}^{1}(\mathbb{R}^n)$. The solution $X = X(t)$ to the gradient flow~\eqref{eqn: grad_flow} satisfies
\begin{equation*}
\label{eqn: gradient_flow_convex1}
f(X) - f(x^{\star}) \leq \frac{ \left( f(x_0) - f(x^\star) \right)  \left\| x_{0} - x^{\star}   \right\|^{2} }{ t  \left( f(x_0) - f(x^\star) \right) +  \left\| x_{0} - x^{\star}\right\|^2}.
\end{equation*}
\end{thm}
\begin{proof}
The time derivative of the distance function  is
\begin{align*}
\frac{\dd}{\dd t} \left\| X - x^{\star} \right\|^{2}          & =         2 \left\langle \dot{X}, X - x^{\star} \right\rangle          \\
                                                                        & =       - 2 \left\langle \nabla f(X),  X - x^{\star} \right\rangle    \\
                                                                        & \leq     0.
\end{align*}
We define a Lyapunov function as 
\[
\mathcal{E} = f(X) - f(x^\star).
\]
With the basic convex inequality for $f \in \mathcal{F}_{L}^{1}(\mathbb{R}^n)$, we have
\[
f(X) - f(x^{\star}) \leq \left\langle \nabla f(X), X - x^{\star} \right\rangle \leq \left\| \nabla f(X) \right\| \left\| x_0 - x^{\star}   \right\|.
\]
Furthermore, we obtain that the time derivative is
\begin{align*} 
\frac{\dd \mathcal{E}}{\dd t}  =          \left\langle \nabla f(X),  \dot{X} \right\rangle   =  - \left\| \nabla f(X) \right\|^{2}   \leq    - \frac{ \left(f(X) - f(x^{\star}) \right)^{2}}{ \left\| x_{0} - x^{\star}   \right\|^{2} }  = - \frac{\mathcal{E}^{2}}{ \left\| x_{0} - x^{\star}   \right\|^{2} }.                                                                     
\end{align*}
Hence,  the convergence rate is 
\[
f(X) - f(x^{\star}) \leq \frac{ \left( f(x_0) - f(x^\star) \right)  \left\| x_{0} - x^{\star}   \right\|^{2} }{ t  \left( f(x_0) - f(x^\star) \right) +  \left\| x_{0} - x^{\star}\right\|^2}.
\] 
\end{proof}

The following theorem is based on the Lyapunov function for gradient flow~\eqref{eqn: grad_flow} in~\cite{su2016differential}.
\begin{thm}
\label{thm: grad_flow_convex2}
Let $f \in \mathcal{F}_{L}^{1}(\mathbb{R}^n)$. The solution $X = X(t)$ to the gradient flow~\eqref{eqn: grad_flow} satisfies
\begin{equation*}
\label{eqn: gradient_flow_convex2}
\left\{ \begin{aligned}
            & f(X) - f(x^{\star}) \leq \frac{ \left\| x_0 - x^{\star} \right\|^{2}}{2t} \\
            & \min_{0 \leq u \leq t} \left\| \nabla f(X(u)) \right\|^{2} \leq \frac{\left\| x_0 - x^{\star} \right\|^{2}}{t^{2}}.
          \end{aligned} \right.  
\end{equation*}
\end{thm}
\begin{proof}
The Lyapunov function is 
\[
\mathcal{E} = t \left( f(X) - f(x^{\star}) \right) + \frac{1}{2} \left\| X - x^{\star}   \right\|^{2}.
\]
We calculate its time derivative as  
\begin{align*}
\frac{\dd \mathcal{E}}{\dd t}       & =  f(X) - f(x^{\star}) + t  \left\langle \nabla f(X),  \dot{X} \right\rangle +  \left\langle X - x^{\star},  \dot{X} \right\rangle          \\
                                          & =  f(X) - f(x^{\star})  -  \left\langle \nabla f(X), X - x^{\star} \right\rangle   - t \left\| \nabla f(X) \right\|_{2}^{2} \\
                                          & \leq - t \left\| \nabla f(X) \right\|^{2}.
\end{align*}
Thus, we complete the proof.
\end{proof}

\begin{rem}
\label{rem: 1}
From the view of Lyapunov function, Theorem~\ref{thm: grad_flow_convex2} is essentially different from Theorem~\ref{thm: grad_flow_convex1}. When the Lyapunov function 
\[
\mathcal{E} = t \left( f(X) - f(x^{\star}) \right) + \frac{1}{2} \left\| X - x^{\star}   \right\|^{2}
\]
is used to take place of that
\[
\mathcal{E} =  f(X) - f(x^{\star}), 
\]
the same convergence rate for function value is not only obtained by the simple way of calculation, but we can also capture an advanced faster speed of the squared gradient norm. From this view, constructing Lyapunov function is a more powerful and advanced mathematical tool for optimization.
\end{rem}

\subsection{Explicit Euler scheme}
\label{sec: discrete_gd_ex}

The corresponding explicit-scheme version  of Theorem~\ref{thm: grad_flow_strongly} is just Theorem 2.1.15 in~\cite{nesterov2013introductory}. We state it below.
\begin{thm}[Theorem 2.1.15,~\cite{nesterov2013introductory}]
\label{thm: grad_descent_strongly}
Let $f \in \mathcal{S}_{\mu, L}^{1}(\mathbb{R}^n)$. Taking any step size $0 < s \leq 2/ \left( \mu + L \right)$, the iterates $\left\{ x_k \right\}_{k = 0}^{\infty}$ generated by \gd~\eqref{eqn:gd} satisfy
\begin{equation*}
\label{eqn: gradient_descent_strongly}
\left\| x_{k} - x^{\star} \right\|^{2} \leq \left(1 - \frac{2 \mu L s}{\mu + L}\right) \left\| x_{0} - x^{\star} \right\|^{2}.
\end{equation*}
In addition, if the step size is set to $s = 2/(\mu + L)$, we get
\begin{equation*}
\label{eqn: gradient_descent_strongly1}
\left\| x_{k} - x^{\star} \right\|^{2} \leq \left(\frac{L - \mu}{L + \mu}\right)^{2} \left\| x_{0} - x^{\star} \right\|^{2}.
\end{equation*}
\end{thm}
This proof is from~\cite{nesterov2013introductory}. The only conceptual difference is that we use the Lyapunov function
\[
\mathcal{E}(k) = \left\| x_{k} - x^{\star} \right\|^{2},
\]                                     
instead of the distance function $r_{k}$ in~\cite{nesterov2013introductory}.

The corresponding explicit version of Theorem~\ref{thm: grad_flow_convex1} is Theorem 2.1.14 in~\cite{nesterov2013introductory}. We also state it as follows.
\begin{thm}[Theorem 2.1.14,~\cite{nesterov2013introductory}]
\label{thm: grad_descent_convex1}
Let $f \in \mathcal{F}_{L}^{1}(\mathbb{R}^n)$. Taking any step size $0 < s < 2/L$, the iterates $\left\{ x_k \right\}_{k = 0}^{\infty}$ generated by \gd~\eqref{eqn:gd} satisfy
\begin{equation}
\label{eqn: gradient_descent_convex_1}
f(x_{k}) - f(x^{\star}) \leq \frac{2 \left( f(x_{0}) - f(x^{\star}) \right)\left\| x_{0} - x^{\star}   \right\|^{2} }{2 \left\| x_{0} - x^{\star}   \right\|^{2} + ks(2 - Ls) \left( f(x_{0}) - f(x^{\star}) \right)}.
\end{equation}
In addition, if the step size is set to $s = 1/L$, we get
\begin{equation}
\label{eqn: gradient_descent_convex_11}
f(x_{k}) - f(x^{\star}) \leq \frac{2L\left\| x_{0} - x^{\star}   \right\|^{2}}{k + 4}.
\end{equation}
\end{thm}

Again, \cite{nesterov2013introductory} uses the Lyapunov function $\mathcal{E}(k)$ instead of $r_k$.

Finally, we show the corresponding discrete version of Theorem~\ref{thm: grad_flow_convex2}, highlighting the ODE-based approach and the importance of Lyapunov functions in proofs.
\begin{thm}
\label{thm: grad_descent_convex3}
Let $f \in \mathcal{F}_{L}^{1}(\mathbb{R}^n)$.
Taking any step size $0 < s \leq 1/L$, the iterates $\{x_{k}\}_{k = 0}^{\infty}$ generated by \gd~\eqref{eqn:gd} satisfy 
\begin{equation}
\label{eqn: gradient_descent_convex_2}
\left\{ \begin{aligned}
          & f(x_{k}) - f(x^{\star})                                                  \leq     \frac{ \left\| x_{0} - x^{\star} \right\|^{2} }{2ks} \\
          & \min_{0 \leq i \leq k}\left\| \nabla f(x_i) \right\|^{2}     \leq     \frac{ 2\left\| x_{0} - x^{\star} \right\|^{2} }{s^2(k + 1)(k + 2)}.
          \end{aligned} \right.
\end{equation}
In addition, if the step size is set $s = 1/L$, we have
\begin{equation}
\label{eqn: gradient_descent_convex_21}
\left\{ \begin{aligned}
          & f(x_{k}) - f(x^{\star}) \leq \frac{L\left\| x_{0} - x^{\star}   \right\|^{2}}{2k} \\
          &  \min_{0 \leq i \leq k}\left\| \nabla f(x_i) \right\|^{2}     \leq     \frac{ 2 L^2 \left\| x_{0} - x^{\star} \right\|^{2} }{(k + 1)(k + 2)}.
          \end{aligned} \right. 
\end{equation}
\end{thm}

To obtain this result, we use a Lyapunov function that is different from the standard analysis of gradient descent, which uses the Lyapunov function $\mathcal{E}(k) \triangleq f(x_k) - f(x^\star)$. This Lyapunov function yields the $O(L/k)$ convergence rate for the function value. For the squared gradient norm, however, this Lyapunov function can only exploit the $L$-smoothness property that transforms the function value to the gradient norm, giving the sub-optimal $O(L^2/k)$ rate, due to the absence of gradient information in this function. Our proof uses a different Lyapunov function: $\mathcal{E}(k) = ks \left( f(x_{k}) - f(x^{\star}) \right) + \frac{1}{2} \left\| x_{k} - x^{\star}   \right\|^{2}.$


 \begin{proof}
 The corresponding discrete Lyapunov function is constructed as below
 \[
 \mathcal{E}(k) = ks \left( f(x_{k}) - f(x^{\star}) \right) + \frac{1}{2} \left\| x_{k} - x^{\star}   \right\|^{2},
 \]
from which we get
 \begin{align*}
         &    \mathcal{E}(k + 1) - \mathcal{E}(k)  \\
 =       &   s \left( f(x_{k }) - f(x^{\star}) \right) + (k + 1) s  \left( f(x_{k + 1}) - f(x_{k}) \right) +  \frac{1}{2}\left\langle x_{k + 1} - x_{k}, x_{k + 1} + x_{k} - 2x^{\star} \right\rangle          \\
 \leq    &  s \left( f(x_{k }) - f(x^{\star})  -  \left\langle \nabla f(x_{k}), x_{k} - x^{\star} \right\rangle \right)  + (k + 1)s\left\langle \nabla f(x_{k}), x_{k+1} - x_{k} \right\rangle \\
          &  + \left[ \frac{(k + 1) s L}{2} + \frac{1}{2}\right]\left\| x_{k + 1} - x_{k}\right\|^{2} \\
 \leq    &  s^{2} \left[ - \frac{1}{2Ls} - (k + 1) +  \frac{(k + 1) s L}{2} + \frac{1}{2}  \right] \left\| \nabla f(x_{k}) \right\|^{2} \\
 \leq    &  - \frac{s^{2}}{2}  \left( k + 1 \right)  \left\| \nabla f(x_{k}) \right\|^{2} 
 \end{align*}
 Taking $k_{0}$ in the assumption completes the proof.
 \end{proof}

\begin{rem}
Same as the continuous ODE in Remark~\ref{rem: 1}, from view of the discrete algorithm, we can find the apunov function is a more powerful and advanced mathematical tool.
\end{rem}


\subsection{Implicit Euler scheme}
\label{sec: discrete_gd_im}

Next, we consider the implicit Euler scheme of the gradient flow~\eqref{eqn: grad_flow} as
\begin{equation}
    \label{eqn: im_grad}
    x_{k+1} = x_{k} - s\nabla f(x_{k+1}), 
\end{equation}
with any initial $x_0 \in \mathbb{R}^n$. The corresponding implicit version of Theorem~\ref{thm: grad_flow_strongly} is shown as below.  
\begin{thm}
\label{thm: grad_descent_implicit_strongly}
 Let $f \in \mathcal{S}_{\mu, L}^{1}(\mathbb{R}^n)$,  the iterates $\{x_{k}\}_{k=0}^{\infty}$ generated by implicit gradient descent~\eqref{eqn: im_grad} satisfy
 \begin{equation}
\label{eqn: gradient_descent_implicit_strongly}
\left\| x_{k} - x^{\star} \right\| \leq \frac{1}{ \left( 1 + \mu s \right)^{k} } \cdot \left\| x_{0} - x^{\star} \right\|.
\end{equation}
In addition, if the step size $s = \theta/\mu$, where $\theta > 0$, we have
\begin{equation}
\label{eqn: gradient_descent_implicit_strongly1}
\left\| x_{k} - x^{\star} \right\| \leq \frac{1}{ \left( 1 + \theta \right)^{k} } \left\| x_{0} - x^{\star} \right\|.
\end{equation}
\end{thm}
\begin{proof}
The Lyapunov function is  
\[
\mathcal{E}(k) = \left\| x_{k} - x^{\star} \right\|^{2}.
\]
Then, we calculate the iterate difference as 
\begin{align*}
\mathcal{E}(k + 1) - \mathcal{E}(k)      & =         \left\| x_{k + 1} - x^{\star} \right\|^{2}  - \left\| x_{k} - x^{\star} \right\|^{2}         \\
                                                             & =        \left\langle x_{k + 1} - x_{k}, x_{k + 1} + x_{k} - 2x^{\star} \right\rangle    \\
                                                             & =         - 2s \left\langle \nabla f(x_{k + 1}),  x_{k + 1} - x^{\star} \right\rangle - s^{2} \left\| \nabla f(x_{k + 1})\right\|^{2} \\
                                                             & \leq     - \left( 2 \mu s + \mu^2 s^2 \right)  \mathcal{E}(k + 1).
\end{align*}  
Hence, the proof is complete.                                                    
\end{proof}

Next, we show the implicit version of Theorem~\ref{thm: grad_flow_convex1} as follows.
\begin{thm}
\label{thm: grad_descent_implicit_convex1}
Let $f \in \mathcal{F}_{L}^{1}(\mathbb{R}^n)$. The iterates $\{x_{k}\}_{k=0}^{\infty}$ generated by implicit gradient descent~\eqref{eqn: im_grad} satisfy
\begin{equation}
\label{eqn: gradient_descent_implicit_convex1}
f(x_{k}) - f(x^{\star}) \leq \frac{(1 + Ls)^{2}  \left( f(x_{0}) - f(x^{\star}) \right)\left\| x_{0} - x^{\star}   \right\|^{2}}{(1 + Ls)^{2} \left\| x_{0} - x^{\star}   \right\|^{2} + ks \left( f(x_{0}) - f(x^{\star}) \right)}.
\end{equation}
In addition, if the step size is set to $s = \theta/L$, we have
\begin{equation}
\label{eqn: gradient_descent_implicit_convex2}
f(x_{k}) - f(x^{\star}) \leq \frac{  L\left\| x_{0} - x^{\star}   \right\|^{2} }{2  + k \cdot \frac{1}{\theta + \frac{1}{\theta} +2 }}.
\end{equation}
\end{thm}
\begin{proof}
Note that the distance function $\left\| x_{k} - x^{\star} \right\|^{2}$ decreases with the iteration number $k$ as 
\begin{align*}
 \left\| x_{k + 1}- x^{\star} \right\|^{2}  -  \left\| x_{k}- x^{\star} \right\|^{2}      & =   - 2s \left\langle \nabla f(x_{k + 1}),  x_{k + 1} - x^{\star} \right\rangle - s^{2} \left\| \nabla f(x_{k + 1})\right\|^{2}       \\
                                                                                                                    & \leq   - s \left( \frac{2}{L} + s \right)   \left\| \nabla f(x_{k + 1})\right\|^{2}  \\
                                                                        & \leq     0.
\end{align*}
With the basic convex inequality for $f \in \mathcal{F}_{L}^{1}(\mathbb{R}^n)$, we have
\[
f(x_{k + 1}) - f(x^{\star})  \leq \left\langle \nabla f(x_{k + 1}), x_{k + 1} - x^{\star} \right\rangle \leq \left\| \nabla f(x_{k + 1}) \right\| \cdot \left\| x_{0} - x^{\star}   \right\|. 
\]
Now, the Lyapunov function is defined as 
\[
\mathcal{E}(k) = f(x_{k}) - f(x^\star).
\]
Then we calculate the difference at the $k$th-iteration as
\begin{align*}
 \mathcal{E}(k + 1) - \mathcal{E}(k)                & =        \left( f(x_{k + 1}) - f(x^{\star}) \right) - \left( f(x_{k}) - f(x^{\star}) \right) \\      
                                                                        & \geq  \left\langle \nabla f(x_{k + 1}),  x_{k + 1} - x_{k} \right\rangle - \frac{L}{2} \left\| x_{k + 1} - x_{k} \right\|^{2} \\
                                                                        & \geq  - s    \left( 1 + \frac{Ls}{2} \right)  \left\| \nabla f(x_{k + 1}) \right\|^{2}   \\
                                                                        & \geq  -2Ls \left( 1 + \frac{Ls}{2} \right) \mathcal{E}(k + 1)   
\end{align*}
and
\begin{align*} 
 \mathcal{E}(k + 1) - \mathcal{E}(k)                & =        \left( f(x_{k + 1}) - f(x^{\star}) \right) - \left( f(x_{k}) - f(x^{\star}) \right) \\      
                                                                        & \leq    \left\langle \nabla f(x_{k + 1}),  x_{k + 1} - x_{k} \right\rangle             \\
                                                                        & =        - s \cdot   \left\| \nabla f(x_{k + 1}) \right\|^{2}                                    \\
                                                                        & \leq    - s \cdot   \frac{ \mathcal{E}(k + 1)^{2} }{\left\| x_{0} - x^{\star}   \right\|_{2}^{2}}   \\
                                                                        & \leq    - s \cdot  \frac{\mathcal{E}(k + 1)}{\mathcal{E}(k)} \cdot \frac{ \mathcal{E}(k) \mathcal{E}(k + 1) }{\left\| x_{0} - x^{\star}   \right\|^{2}}  \\
                                                                        & \leq    - \frac{s}{ (1 + Ls)^{2}}    \cdot \frac{ \mathcal{E}(k) \mathcal{E}(k + 1) }{\left\| x_{0} - x^{\star}   \right\|^{2}}.                                                            
\end{align*}
Hence,  the convergence rate is given as
\[
f(x_{k}) - f(x^{\star}) \leq \frac{(1 + Ls)^{2}  \left( f(x_{0}) - f(x^{\star}) \right)\left\| x_{0} - x^{\star}   \right\|^{2}}{(1 + Ls)^{2} \left\| x_{0} - x^{\star}   \right\|^{2} + ks \left( f(x_{0}) - f(x^{\star}) \right)}.
\] 
\end{proof}

Finally, we present the implicit version of Theorem~\ref{thm: grad_flow_convex2}.
\begin{thm}
\label{thm: grad_descent_implicit_convex2}
Let $f \in \mathcal{F}_{L}^{1}(\mathbb{R}^n)$. The iterates  $\{x_{k}\}_{k=0}^{\infty}$ generated by implicit gradient descent~\eqref{eqn: im_grad} satisfy 
\begin{equation}
\label{eqn: gradient_descent_implicit_convex3}
\left\{ \begin{aligned}
          & f(x_{k}) - f(x^{\star})                                                  \leq     \frac{ \left\| x_{0} - x^{\star} \right\|^{2} }{2ks} \\
          & \min_{0 \leq i \leq k}\left\| \nabla f(x_i) \right\|^{2}     \leq     \frac{ 2\left\| x_{0} - x^{\star} \right\|^{2} }{s^2(k + 1)(k + 2)}.
          \end{aligned} \right.
\end{equation}
In addition, if the step size is set to $s = 1/L$, we have
\begin{equation}
\label{eqn: gradient_descent_implicit_convex4}
\left\{ \begin{aligned}
          & f(x_{k}) - f(x^{\star})                                                  \leq     \frac{L \left\| x_{0} - x^{\star} \right\|^{2} }{2k} \\
          & \min_{0 \leq i \leq k}\left\| \nabla f(x_i) \right\|^{2}     \leq     \frac{ 2L^2 \left\| x_{0} - x^{\star} \right\|^{2} }{(k + 1)(k + 2)}.
          \end{aligned} \right.
\end{equation}
\end{thm}
\begin{proof}
The Lyapunov function is 
\[
\mathcal{E}(k) = ks \left( f(x_{k}) - f(x^{\star}) \right) + \frac{1}{2} \left\| x_{k} - x^{\star}   \right\|^{2}.
\]
Then, we calculate the iterate difference as
\begin{align*}
\lefteqn{\mathcal{E}(k + 1) - \mathcal{E}(k)}  \\
& =         s \left( f(x_{k + 1}) - f(x^{\star}) \right) + ks \left( f(x_{k + 1}) - f(x_{k}) \right) +  \frac{1}{2}\left\langle x_{k + 1} - x_{k}, x_{k + 1} + x_{k} - 2x^{\star} \right\rangle          \\
& \leq     s \left( f(x_{k + 1}) - f(x^{\star})  -  \left\langle \nabla f(x_{k + 1}), x_{k + 1} - x^{\star} \right\rangle \right)  \\
           & \quad + k s \left\langle \nabla f(x_{k + 1}), x_{k+1} - x_{k} \right\rangle  - \frac{1}{2}\left\| x_{k + 1} - x_{k}\right\|^{2} \\
& \leq     - s^{2} \left(  \frac{1}{2Ls} + k + \frac{1}{2}  \right) \left\| \nabla f(x_{k + 1}) \right\|^{2} \\
& \leq     - \frac{s^{2}}{2} (k + 1) \left\| \nabla f(x_{k + 1}) \right\|^{2}.
\end{align*}
Hence, the proof is complete.
\end{proof}


\section{Proofs for Section~\ref{sec: high_res}}
\label{sec: appendix_high_res}
Here, we first describe and prove Theorem~\ref{thm: high_nag-sc-three} below. Then we complete the proof of Theorem~\ref{thm:fixed-s-nag-three} by viewing it as a special case of Theorem~\ref{thm: high_nag-sc-three}.

\begin{thm}[Discretization of NAG-\texttt{SC} ODE --- General]\label{thm: high_nag-sc-three}
For any $f \in \mathcal{S}_{\mu, L}^{1}(\mathbb{R}^n)$, the following conclusions hold:
\begin{enumerate}[leftmargin=*,topsep=0pt]
\item[(a)] Taking $0 < s \leq 4/(9L)$, the symplectic Euler scheme satisfies
\begin{multline}\label{eqn: conver_nag-sc_sym_high_fix}
f(x_k) - f(x^\star) \\
 \leq \left( \frac{sL \left( 2 + (1 + 3\sqrt{\mu s})^{2} \right)}{(1 + \sqrt{\mu s})^{2}} + \frac{2\mu}{L} + \frac{1 + \sqrt{\mu s}}{2} - \frac{sL (1 + \sqrt{\mu s})^{2}}{2 (1 + 2\sqrt{\mu s})} \right)  \frac{L\left\| x_0 - x^\star \right\|^2}{\left( 1 + \frac{\sqrt{\mu s}}{6} \right)^{k}}.                                       
 \end{multline}

\item[(b)] Taking $0 < s \leq \mu/(100L^{2})$, the explicit Euler scheme satisfies
\begin{multline}\label{eqn: conver_nag-sc_ex_high_fix}
f(x_k) - f(x^\star) \\
\leq \left( \frac{3 - 2\sqrt{\mu s} + \mu s}{2 + 4 \sqrt{\mu s} + 2\mu s} \cdot sL +  \frac{2\mu}{L} + \frac{1 + \sqrt{\mu s}}{2}  \right)  L\left\| x_0 - x^\star \right\|^2 \left( 1 - \frac{\sqrt{\mu s}}{8} \right)^{k}.                                       
\end{multline}

\item[(c)] Taking $0 < s \leq 1/L$, the implicit Euler scheme satisfies
\begin{equation}
\label{eqn: conver_nag-sc_im_high_fix}
f(x_k) - f(x^\star) \leq \left( \frac{3 - 2\sqrt{\mu s} + \mu s}{2 + 4 \sqrt{\mu s} + 2\mu s} \cdot sL +  \frac{2\mu}{L} + \frac{1 + \sqrt{\mu s}}{2}  \right) \frac{ L\left\| x_0 - x^\star \right\|^2}{  \left( 1 + \frac{\sqrt{\mu s}}{4} \right)^{k} }.                                       
 \end{equation} 
\end{enumerate}
 \end{thm}
\subsection{Proof of Theorem~\ref{thm: high_nag-sc-three}}
\label{subsec: high-nag-sc}

\begin{enumerate}
\item [(a)]

The Lyapunov function is constructed as 
\begin{multline*}
\mathcal{E}(k) =     \frac{1}{4} \left\| v_{k} \right\|^{2} + \frac{1}{4} \left\| 2\sqrt{\mu}(x_{k + 1} - x^\star) + v_{k}  + \sqrt{s} \nabla f(x_{k}) \right\|^{2} \\
                              + \left( 1 + \sqrt{\mu s} \right) \left( f(x_{k}) - f(x^\star) \right) - \frac{( 1 + \sqrt{\mu s} )^{2}}{1 + 2 \sqrt{\mu s}} \cdot \frac{s}{2} \left\| \nabla f(x_{k}) \right\|^{2}.
\end{multline*}
With the basic inequality for $f \in \mathcal{S}_{\mu, L}^{1}(\mathbb{R}^n)$
\[
f(x_{k + 1}) - f(x_{k}) \leq \left\langle \nabla f(x_{k + 1}), x_{k + 1} - x_{k} \right\rangle - \frac{1}{2L} \left\| \nabla f(x_{k + 1}) - \nabla f(x_{k}) \right\|^{2},
\]
then the iterate difference can be calculated as
\begin{align*}
\lefteqn{ \mathcal{E}(k + 1) - \mathcal{E}(k)}  \\ 
& =  \frac{1}{4} \left\langle v_{k + 1} - v_{k}, v_{k + 1} + v_{k} \right\rangle + \left( 1 + \sqrt{\mu s} \right) \left( f(x_{k + 1}) - f(x_{k}) \right) \\
& \mathrel{\phantom{=}}   
 \begin{aligned}
+  \frac{1}{4}\langle    & 2\sqrt{\mu} (x_{k + 2} - x_{k + 1}) +v_{k + 1} - v_{k}  +  \sqrt{s} \left( \nabla f(x_{k + 1}) - \nabla f(x_{k}) \right),  \\                                                                    
                                   &\begin{aligned}
                                   2 \sqrt{\mu} \left( x_{k + 2} + x_{k + 1} - 2x^\star  \right) &+  v_{k + 1} + v_{k}   \\ 
                                                                                                                           &+ \sqrt{s} \left( \nabla f(x_{k + 1}) + \nabla f(x_{k}) \right)  \rangle   
                                      \end{aligned}
\end{aligned}  \\
& \mathrel{\phantom{=}}   -  \frac{( 1 + \sqrt{\mu s} )^{2}}{1 + 2 \sqrt{\mu s}} \cdot \frac{s}{2} \left( \left\| \nabla f(x_{k + 1}) \right\|^{2} - \left\| \nabla f(x_{k}) \right\|^{2} \right)   \\                                         
&  \leq       -  \sqrt{\mu s} \left\| v_{k + 1} \right\|^{2} - \frac{\sqrt{s}}{2 (1 + \sqrt{\mu s})} \left\langle \nabla f(x_{k + 1}) - \nabla f(x_{k}), v_{k} \right\rangle \\
&\mathrel{\phantom{\leq}}  +  \frac{s}{2 \left( 1 + 2 \sqrt{\mu s} \right)} \left\| \nabla f(x_{k + 1}) - \nabla f(x_{k}) \right\|^{2} \\ 
&\mathrel{\phantom{\leq}} +  \frac{s}{2} \cdot \frac{1 + \sqrt{\mu s}}{1 + 2 \sqrt{\mu s}} \left\langle \nabla f(x_{k + 1}) - \nabla f(x_{k}), \nabla f(x_{k + 1}) \right\rangle \\
&\mathrel{\phantom{\leq}} -  \frac{\sqrt{s} \left(1 + \sqrt{\mu s} \right)}{2} \left\langle \nabla f(x_{k + 1}), v_{k + 1} \right\rangle - \frac{1}{4} \left\| v_{k + 1} - v_{k} \right\|^{2}   \\
&\mathrel{\phantom{\leq}} + \left( 1 + \sqrt{\mu s} \right) \sqrt{s} \left\langle \nabla f(x_{k + 1}), v_{k} \right\rangle  - \frac{1 + \sqrt{\mu s}}{2L} \left\| \nabla f(x_{k + 1}) - \nabla f(x_{k}) \right\|^{2}     \\
&\mathrel{\phantom{\leq}}\begin{aligned} 
                                         -\frac{1}{2} \langle &(1 + \sqrt{\mu s})\sqrt{s} \nabla f(x_{k + 1}), \\
                                                                       &\left( 1 + 2\sqrt{\mu s} \right) v_{k + 1} + 2\sqrt{\mu} (x_{k + 1} - x^\star) + \sqrt{s} \nabla f(x_{k + 1}) \rangle  
                                           \end{aligned}   \\
&\mathrel{\phantom{\leq}} -  \frac{1}{4} \left( 1 + \sqrt{\mu s} \right)^{2} s \left\| \nabla f(x_{k + 1}) \right\|^{2}  -   \frac{s}{2} \left( \left\| \nabla f(x_{k + 1}) \right\|^{2} - \left\| \nabla f(x_{k}) \right\|^{2} \right)  \\
& \leq        -  \sqrt{\mu s}  \left( \left\| v_{k + 1} \right\|^{2} + \left(1 + \sqrt{\mu s}\right) \left\langle \nabla f(x_{k + 1}), x_{k + 1} - x^\star \right\rangle \right)   \\
 &\mathrel{\phantom{\leq}}    -  \left( \frac{1 + \sqrt{\mu s}}{2} \right) \left[ \sqrt{s} \left\langle \nabla f(x_{k + 1}), \left( 1 + 2\sqrt{\mu s} \right) v_{k + 1} - v_{k} \right\rangle + s \left\| \nabla f(x_{k + 1}) \right\|^{2} \right]  \\
 &\mathrel{\phantom{\leq}}  - \frac{\sqrt{s}}{2 (1 + \sqrt{\mu s})} \left\langle \nabla f(x_{k + 1}) - \nabla f(x_{k}), v_{k} \right\rangle  \\
 &\mathrel{\phantom{\leq}}  + \frac{s}{2} \cdot \frac{1 + \sqrt{\mu s}}{1 + 2 \sqrt{\mu s}} \left\langle \nabla f(x_{k + 1}) - \nabla f(x_{k}), \nabla f(x_{k + 1}) \right\rangle    \\
 &\mathrel{\phantom{\leq}}   \begin{aligned}  
                                               - \frac{1}{4} \left[ \left\| v_{k + 1} - v_{k} \right\|^{2} \right. &+   \left( 1 + \sqrt{\mu s} \right)^{2} s \left\| \nabla f(x_{k + 1}) \right\|^{2} \\
                                                                  &  + \left. 2 (1 + \sqrt{\mu s}) \sqrt{s} \left\langle \nabla f(x_{k + 1}), v_{k + 1} - v_{k} \right\rangle \right]      
                                             \end{aligned}     \\
 &\mathrel{\phantom{\leq}}       - \frac{1}{2} \left( \frac{1 + \sqrt{\mu s} }{L} - \frac{s}{1 + 2 \sqrt{\mu s}} \right) \left\| \nabla f(x_{k + 1}) - \nabla f(x_{k}) \right\|^{2}    \\
 &\mathrel{\phantom{\leq}}       - \frac{( 1 + \sqrt{\mu s} )^{2}}{1 + 2 \sqrt{\mu s}} \cdot  \frac{s}{2} \left( \left\| \nabla f(x_{k + 1}) \right\|^{2} - \left\| \nabla f(x_{k}) \right\|^{2} \right).                                                                                                 
\end{align*}

Noting the following two inequalities
\begin{align*}
\lefteqn{\begin{aligned}
- \frac{1}{4} \left[ \left\| v_{k + 1} - v_{k} \right\|^{2} \right. &+ \left( 1 + \sqrt{\mu s} \right)^{2} s \left\| \nabla f(x_{k + 1}) \right\|^{2} \\ 
             &+ \left. 2 (1 + \sqrt{\mu s}) \sqrt{s} \left\langle \nabla f(x_{k + 1}), v_{k + 1} - v_{k} \right\rangle   \right] 
\end{aligned}}    \\         
&\mathrel{\phantom{=}} =       - \frac{1}{4} \left\| v_{k + 1} - v_{k} + \left( 1 + \sqrt{\mu s} \right) \sqrt{s} \nabla f(x_{k})  \right\|^{2}   \leq 0,   
\end{align*}
and 
\begin{align*}
     & \lefteqn{ -  \frac{1}{2} \left( 1 + \sqrt{\mu s} \right) \left[ \sqrt{s} \left\langle \nabla f(x_{k + 1}), \left( 1 + 2\sqrt{\mu s} \right) v_{k + 1} - v_{k} \right\rangle + s \left\| \nabla f(x_{k + 1}) \right\|^{2} \right] } \\
&   \begin{aligned}  
       = -  \left( \frac{1 + \sqrt{\mu s}}{2} \right) [ & \sqrt{s} \langle \nabla f(x_{k + 1}), \\
       &\mathrel{\phantom{ \sqrt{s} \langle}}-\sqrt{s} \left( \nabla f(x_{k + 1}) - \nabla f(x_{k}) \right)  - \sqrt{s} (1 + \sqrt{\mu s}) \nabla f(x_{k + 1})
       \rangle \\
    &  + s \left\| \nabla f(x_{k + 1}) \right\|^{2} ] 
         \end{aligned} \\
&=      \frac{\left( 1 + \sqrt{\mu s} \right) s}{2} \left( \left\langle \nabla f(x_{k + 1}) - \nabla f(x_{k}), \nabla f(x_{k + 1}) \right\rangle + \sqrt{\mu s}  \left\| \nabla f(x_{k + 1}) \right\|^{2} \right),
\end{align*}
we see that the iterate difference is
\begin{align*}
& \lefteqn{\mathcal{E}(k + 1) - \mathcal{E}(k)} \\
   & \leq           - \sqrt{\mu s} \left[ \left\| v_{k + 1} \right\|^{2} + \left(1 + \sqrt{\mu s} \right) \left( \left\langle \nabla f(x_{k + 1}), x_{k + 1} - x^\star \right\rangle   - \frac{s}{2} \left\| \nabla f(x_{k + 1}) \right\|^{2} \right) \right]  \\
                                                                   & \mathrel{ \phantom{\leq} }     - \frac{1}{2 (1 + \sqrt{\mu s})} \left\langle \nabla f(x_{k + 1}) - \nabla f(x_{k}), x_{k + 1} - x_{k} \right\rangle   \\
                                                                   & \mathrel{ \phantom{\leq} }      + \frac{( 1 + \sqrt{\mu s} )^{2}}{1 + 2 \sqrt{\mu s}} \cdot s \left\langle \nabla f(x_{k + 1}), \nabla f(x_{k + 1}) - \nabla f(x_{k}) \right\rangle  \\
                                                                   &  \mathrel{ \phantom{\leq} }       - \frac{1}{2} \left( \frac{1 + \sqrt{\mu s} }{L} - \frac{s}{1 + 2 \sqrt{\mu s}} \right) \left\| \nabla f(x_{k + 1}) - \nabla f(x_{k}) \right\|^{2}    \\
                                                                    & \mathrel{ \phantom{\leq} }      -  \frac{( 1 + \sqrt{\mu s} )^{2}}{1 + 2 \sqrt{\mu s}} \cdot \frac{s}{2} \left( \left\| \nabla f(x_{k + 1}) \right\|^{2} - \left\| \nabla f(x_{k}) \right\|^{2} \right)       \\
                                                                    &  \leq      - \sqrt{\mu s} \left[ \left\| v_{k + 1} \right\|^{2} + \left(1 + \sqrt{\mu s} \right) \left( \left\langle \nabla f(x_{k + 1}), x_{k + 1} - x^\star \right\rangle   - \frac{s}{2} \left\| \nabla f(x_{k + 1}) \right\|^{2} \right) \right]  \\
                                                                   &  \mathrel{ \phantom{\leq} }       - \frac{1}{2 (1 + \sqrt{\mu s})} \left\langle \nabla f(x_{k + 1}) - \nabla f(x_{k}), x_{k + 1} - x_{k} \right\rangle  \\
                                                                   &  \mathrel{ \phantom{\leq} }    + \frac{1}{2} \left( \frac{1}{1 + 2 \sqrt{\mu s}} + \frac{( 1 + \sqrt{\mu s} )^{2}}{1 + 2 \sqrt{\mu s}} -  \frac{1 + \sqrt{\mu s} }{Ls}  \right) s  \left\| \nabla f(x_{k + 1}) - \nabla f(x_{k}) \right\|^{2}.                             
\end{align*}
Furthermore, taking the basic inequality for $f \in \mathcal{S}_{\mu, L}^{1}(\mathbb{R}^{n})$
\[
\left\langle \nabla f(x_{k + 1}) - \nabla f(x_{k}), x_{k + 1} - x_{k} \right\rangle \geq \frac{1}{L} \left\| \nabla f(x_{k + 1}) - \nabla f(x_{k}) \right\|^{2},
\]
the iterate difference can be calculated as
\begin{align*}
 &\mathcal{E}(k + 1) - \mathcal{E}(k)   \\
 &\leq           - \sqrt{\mu s} \left[ \left\| v_{k + 1} \right\|^{2} + \left(1 + \sqrt{\mu s} \right) \left( \left\langle \nabla f(x_{k + 1}), x_{k + 1} - x^\star \right\rangle   - \frac{s}{2} \left\| \nabla f(x_{k + 1}) \right\|^{2} \right) \right]  \\
  & \mathrel{ \phantom{\leq} } - \frac{2 + 2 \sqrt{\mu s} + \mu s}{2 \left(1 + 2 \sqrt{\mu s} \right)} \left( \frac{1}{L} - s \right) \left\| \nabla f(x_{k + 1}) - \nabla f(x_{k}) \right\|^{2}.
\end{align*}

Next, we consider how to set the step size $s$. First,  when the step size satisfies $s \leq 1/L$, we have
\begin{align*}
& \lefteqn{\mathcal{E}(k + 1) - \mathcal{E}(k)}  \\
 &\leq           - \sqrt{\mu s} \left[ \left\| v_{k + 1} \right\|^{2} + \left(1 + \sqrt{\mu s} \right) \left( \left\langle \nabla f(x_{k + 1}), x_{k + 1} - x^\star \right\rangle   - \frac{s}{2} \left\| \nabla f(x_{k + 1}) \right\|^{2} \right) \right]. 
\end{align*}
Noting the basic inequality for $f \in \mathcal{S}_{\mu, L}^{1}(\mathbb{R}^{n})$
\[
f(x_{k + 1}) - f(x^\star) \leq \left\langle \nabla f(x_{k + 1}), x_{k + 1} - x^\star \right\rangle   - \frac{1}{2L} \left\| \nabla f(x_{k + 1}) \right\|^2,
\]
the iterate difference can be obtained as 
\begin{align*}
&\lefteqn{\mathcal{E}(k + 1) - \mathcal{E}(k)}   \\
&\begin{aligned}   \leq  - \sqrt{\mu s} \left[ \left( f(x_{k + 1}) - f(x^\star) \right) \right.   &+ \left\| v_{k + 1}\right\|^2 + \frac{\mu}{2} \left\| x_{k} - x^\star \right\|^{2} \\
                                                                                                                                   &   +  \left.\sqrt{\mu s} \left( f(x_{k + 1}) - f(x^\star) - \frac{s}{2} \left\| \nabla f(x_{k + 1}) \right\|^{2} \right) \right].
               \end{aligned} 
\end{align*}

Furthermore, using the Cauchy-Schwarz inequality
\begin{align*}
& \lefteqn{\left\| 2 \sqrt{\mu} \left( x_{k + 1} - x^\star \right) + v_{k} + \sqrt{s} \nabla f(x_{k}) \right\|^{2} }  \\
& =     \left\| 2 \sqrt{\mu} (x_{k} - x^\star) + (1 + 2\sqrt{\mu s}) v_{k} + \sqrt{s} \nabla f(x_{k})  \right\|^{2} \\
& \leq 3 \left( 4 \mu \left\| x_{k} - x^\star \right\|^{2} +  (1 + 2\sqrt{\mu s})^2 \left\| v_{k} \right\|^{2} + s \left\| \nabla f(x_{k}) \right\|^{2} \right),
\end{align*}
and the following basic inequality for $f \in \mathcal{S}_{\mu, L}^{1}(\mathbb{R}^n)$ 
\begin{align*}
&\lefteqn{\frac{3s}{4} \left\| \nabla f(x_{k}) \right\|^{2} - \frac{(1 + \sqrt{\mu s})^{2}}{1 + 2 \sqrt{\mu s}} \cdot \frac{s}{2} \left\| \nabla f(x_{k}) \right\|^{2}} \\
& \leq \frac{Ls}{2} \left( f(x_{k}) - f(x^\star) \right) - \frac{\mu s^{2}}{2 \left(1 + 2 \sqrt{\mu s}\right)} \left\| \nabla f(x_{k}) \right\|^{2},
\end{align*}
the Lyapunov function satisfies
\begin{align*}
\mathcal{E}(k) & \leq  \left( 1 + \sqrt{\mu s}  + \frac{Ls}{2}\right) \left( f(x_{k}) - f(x^\star) \right)  + \left( 1 + 3\sqrt{\mu s} + 3 \mu s \right) \left\| 
v_{k} \right\|^2\\
                               & \mathrel{ \phantom{\leq} }  + 3\mu \left\| x_{k} - x^\star \right\|^{2} + \frac{\mu s}{1 + 2\sqrt{\mu s}} \left( f(x_{k}) - f(x^\star) - \frac{s}{2} \left\| \nabla f(x_{k}) \right\|^{2} \right).
\end{align*}
Therefore, when $s \leq 4/ (9L)$, the iterate difference for the Lyapunov function satisfies
\[
\mathcal{E}(k + 1) - \mathcal{E}(k) \leq - \frac{\sqrt{\mu s}}{6} \mathcal{E}(k + 1). 
\]
Hence, the proof is complete.

\item[(b)]
The Lyapunov function is 
\begin{align*}
\mathcal{E}(k) =  \frac{1}{4} \left\| v_{k} \right\|^{2} &+ \left( 1 + \sqrt{\mu s} \right) \left( f(x_{k}) - f(x^\star) \right) \\
                                                                               & + \frac{1}{4} \left\| 2\sqrt{\mu}(x_{k} - x^\star) + v_{k} + \sqrt{s} \nabla f(x_{k}) \right\|^{2}. 
\end{align*}
With the basic inequality for $f \in \mathcal{S}_{\mu, L}^{1}(\mathbb{R}^n)$
\[
f(x_{k + 1}) - f(x_{k}) \leq \left\langle \nabla f(x_{k}), x_{k + 1} - x_{k} \right\rangle + \frac{L}{2} \left\| x_{k + 1} - x_{k} \right\|^{2}, 
\]
we can calculate the iterate difference
\begin{align*}
& \lefteqn{\mathcal{E}(k + 1) - \mathcal{E}(k)}\\
&  =    \frac{1}{4} \left\langle v_{k + 1} - v_{k}, v_{k + 1} + v_{k} \right\rangle + \left( 1 + \sqrt{\mu s} \right) \left( f(x_{k + 1}) - f(x_{k}) \right) \\
&  \mathrel{ \phantom{\leq} }       
\begin{aligned} 
+  \frac{1}{4} \langle &2\sqrt{\mu} (x_{k + 1} - x_{k}) + v_{k + 1} - v_{k} + \sqrt{s} \left( \nabla f(x_{k + 1}) - \nabla f(x_{k}) \right),  \\
                                & 2 \sqrt{\mu} \left( x_{k + 1} + x_{k} - 2x^\star \right) + v_{k + 1} + v_{k} +  \left. \sqrt{s} ( \nabla f(x_{k + 1}) + \nabla f(x_{k}) \right\rangle \end{aligned}  \\
&         \leq    \frac{1}{2} \left\langle v_{k+1} - v_{k}, v_{k} \right\rangle  + \frac{1}{4} \left\| v_{k + 1} - v_{k} \right\|^{2} \\
&   \mathrel{ \phantom{\leq} }      + \left( 1 + \sqrt{\mu s} \right) \left( \left\langle \nabla f(x_{k}), x_{k+ 1} - x_{k} \right\rangle + \frac{L}{2} \left\| x_{k + 1} - x_{k} \right\|^2\right)  \\
&   \mathrel{ \phantom{\leq} }       - \frac{1}{2} \left\langle  (1 + \sqrt{\mu s}) \sqrt{s}\nabla f(x_{k}), 2 \sqrt{\mu} (x_{k} - x^\star) + v_{k} + \sqrt{s} \nabla f(x_{k}) \right\rangle \\
&    \mathrel{ \phantom{\leq} }       + \frac{1}{4} \left\| (1 + \sqrt{\mu s}) \sqrt{s}\nabla f(x_{k}) \right\|^2  \\
&         =       -  \sqrt{\mu s} \left\| v_{k} \right\|^{2}  - \frac{1}{2} \left\langle \nabla f(x_{k + 1}) - \nabla f(x_{k}), x_{k + 1} - x_{k} \right\rangle \\
&  \mathrel{ \phantom{=} }        + \frac{1}{4} \left\| 2\sqrt{\mu s} v_{k} + \sqrt{s} \left( \nabla f(x_{k + 1}) - \nabla f(x_{k}) \right) + \sqrt{s} \left( 1 + \sqrt{\mu s} \right) \nabla f(x_{k}) \right\|^{2}  \\
&  \mathrel{ \phantom{=} }        + \frac{(1 + \sqrt{\mu s})sL}{2} \left\| v_{k} \right\|^{2}  - \sqrt{\mu s}  \left( 1 + \sqrt{\mu s} \right) \left\langle \nabla f(x_{k}), x_{k} - x^\star \right\rangle  \\
&   \mathrel{ \phantom{=} }         - \frac{\left( 1 + \sqrt{\mu s} \right)s}{2} \left\| \nabla f(x_{k}) \right\|^{2} + \frac{1}{4} \left( 1 + \sqrt{\mu s} \right)^{2} s \left\| \nabla f(x_{k}) \right\|^{2}.  
\end{align*}  

Using the Cauchy-Schwartz inequality
\begin{align*}
              \lefteqn{\left\| 2\sqrt{\mu s} v_{k} + \sqrt{s} \left( \nabla f(x_{k + 1}) - \nabla f(x_{k}) \right) + \sqrt{s} \left( 1 + \sqrt{\mu s} \right) \nabla f(x_{k}) \right\|^{2}} \\
  & \leq       12 \mu s \left\| v_{k} \right\|^{2} + 3 s \left\|  \nabla f(x_{k + 1}) - \nabla f(x_{k})  \right\|^{2} + 3s \left( 1 + \sqrt{\mu s} \right)^{2} \left\| \nabla f(x_0) \right\|^{2},            
\end{align*}
the iterate difference for the Lyapunov function can be calculated as 
\begin{align*}
&\lefteqn{\mathcal{E}(k + 1) - \mathcal{E}(k) }  \\
          & \leq             -  \sqrt{\mu s} \left( \left\| v_{k} \right\|^{2} + \left( 1 + \sqrt{\mu s} \right) \left\langle \nabla f(x_{k}), x_{k} - x^\star \right\rangle + \frac{s}{2} \left\| \nabla f(x_{k}) \right\|^{2} \right) \\
                                                                    &  \mathrel{ \phantom{\leq} }      -  \frac{1}{2} \left\langle \nabla f(x_{k + 1}) - \nabla f(x_{k}), x_{k + 1} - x_{k} \right\rangle + \frac{3s}{4} \left\| \nabla f(x_{k + 1}) - \nabla f(x_{k}) \right\|^{2} \\
                                                                    &\mathrel{ \phantom{\leq} }       + \left( 3 \mu s + \frac{(1 + \sqrt{\mu s})sL}{2} \right) \left\| v_{k} \right\|^{2}       +  \left[  \left(1 + \sqrt{\mu s} \right)^{2} - \frac{1}{2} \right] s \left\| \nabla f(x_{k}) \right\|^{2}.
 \end{align*}

Furthermore, combined with the basic inequality for $f \in \mathcal{S}_{\mu, L}^{1}(\mathbb{R}^n)$, 
\[
\left\{ \begin{aligned}
            &  \left\| \nabla f(x_{k + 1}) - \nabla f(x_{k}) \right\|^{2} \leq  L \left\langle \nabla f(x_{k + 1}) - \nabla f(x_{k}), x_{k + 1} - x_{k} \right\rangle \\
            &   \left\| \nabla f(x_{k}) \right\|^{2} \leq  L \left\langle \nabla f(x_{k}), x_{k } - x^\star \right\rangle,
         \end{aligned} \right. 
\]
the iterate difference for the Lyapunov function can be calculated as
\begin{align*}
 &\lefteqn{\mathcal{E}(k + 1) - \mathcal{E}(k)}     \\
   & \leq      -  \frac{\sqrt{\mu s}}{2} \left( \left\| v_{k} \right\|^{2} + \left( 1 + \sqrt{\mu s} \right) \left\langle \nabla f(x_{k}), x_{k} - x^\star \right\rangle + \frac{s}{2} \left\| \nabla f(x_{k}) \right\|^{2}  \right)  \\
                                                               & \mathrel{ \phantom{\leq} }     -  \left( \frac{1}{2L} - \frac{3s}{4}\right) \left\| \nabla f(x_{k + 1}) - \nabla f(x_{k}) \right\|^{2}  \\
                                                               & \mathrel{ \phantom{\leq} }      - \left( \frac{\sqrt{\mu s}}{2} - 3 \mu s - \frac{(1 + \sqrt{\mu s}) sL}{2} \right) \left\| v_{k} \right\|^{2} \\
                                                               & \mathrel{ \phantom{\leq} }      - \left[ \frac{\sqrt{\mu s} \left( 1 + \sqrt{\mu s} \right)}{2L} - \left( \frac{1}{2} +  \sqrt{\mu s}\right) \left( 1 + \sqrt{\mu s}\right) s  \right] \left\| \nabla f(x_{k})\right\|^{2}.
\end{align*}

Simple calculation tells us when the step size satisfies $s \leq \mu / (100 L^{2})$, we have
\begin{multline*}
\mathcal{E}(k + 1) - \mathcal{E}(k)  \\  \leq          -  \frac{\sqrt{\mu s}}{2} \left( \left\| v_{k} \right\|^{2} + \left( 1 + \sqrt{\mu s} \right) \left\langle \nabla f(x_{k}), x_{k} - x^\star \right\rangle + \frac{s}{2} \left\| \nabla f(x_{k}) \right\|^{2} \right).
\end{multline*}
Furthermore, taking the Cauchy-Schwartz inequality
\begin{multline*}
\left\| 2 \sqrt{\mu} \left( x_{k + 1} - x^\star \right) + v_{k} + \sqrt{s} \nabla f(x_{k}) \right\|^{2}   \\ \leq 3 \left( 4 \mu \left\| x_{k} - x^\star \right\|^{2} +  \left\| v_{k} \right\|^{2} + s \left\| \nabla f(x_{k}) \right\|^{2} \right),
\end{multline*}
we can obtain the final estimate for the iterate difference
\[
\mathcal{E}(k + 1) - \mathcal{E}(k)    \leq          -  \frac{\sqrt{\mu s}}{8} \mathcal{E}(k).
\]
Hence, the proof is complete.

\item[(c)]

The Lyapunov function is 
\begin{align*}
\mathcal{E}(k) = \frac{1}{4} \left\| v_{k} \right\|^{2} &+ \left( 1 + \sqrt{\mu s} \right) \left( f(x_{k}) - f(x^\star) \right) \\
                                                                               &+ \frac{1}{4} \left\| 2\sqrt{\mu}(x_k - x^\star) + v_{k} + \sqrt{s} \nabla f(x_{k}) \right\|^{2} 
\end{align*}
With the Cauchy-Schwartz inequality
\begin{multline*}
\left\| 2 \sqrt{\mu} \left( x_{k} - x^\star \right) + v_{k} + \sqrt{s} \nabla f(x_{k}) \right\|^{2} \\ \leq 3 \left( 4 \mu \left\| x_{k} - x^\star \right\|^{2} + \left\| v_{k} \right\|^{2} + s \left\| \nabla f(x_{k}) \right\|^{2} \right),
\end{multline*}
and the basic inequality for $f \in \mathcal{S}_{\mu, L}^{1}(\mathbb{R}^n)$
\[
\left\{ \begin{aligned}
         & f(x_{k + 1}) - f(x_{k}) \leq \left\langle \nabla f(x_{k + 1}), x_{k + 1} - x_{k} \right\rangle  \\
         & f(x^\star) \geq f(x_{k + 1}) + \left\langle \nabla f(x_{k + 1}), x^\star - x_{k + 1} \right\rangle + \frac{\mu}{2} \left\| x_{k + 1} - x^\star \right\|^{2}  \\
         & \left\langle \nabla f(x_{k + 1}) - \nabla f(x_{k}), x_{k + 1} - x_{k} \right\rangle \geq \mu \left\| x_{k + 1} - x_{k} \right\|^{2} \geq 0,
         \end{aligned} \right.
\]
we can calculate the iterate difference
\begin{align*}
& \lefteqn{\mathcal{E}(k + 1) - \mathcal{E}(k) }    \\
& =                   \frac{1}{4} \left\langle v_{k + 1} - v_{k}, v_{k + 1} + v_{k} \right\rangle + \left( 1 + \sqrt{\mu s} \right) \left( f(x_{k + 1}) - f(x_{k}) \right) \\
&  \mathrel{ \phantom{=} }              +  \frac{1}{4} \left\langle 2\sqrt{\mu} (x_{k + 1} - x_{k}) + v_{k + 1} - v_{k} + \sqrt{s} \left( \nabla f(x_{k + 1}) - \nabla f(x_{k}) \right) \right. \\
& \mathrel{ \phantom{=} }    \left. 2 \sqrt{\mu} \left( x_{k + 1} + x_{k} - 2x^\star \right) + v_{k + 1} + v_{k} + \sqrt{s} \left( \nabla f(x_{k + 1}) + \nabla f(x_{k}) \right)  \right\rangle \\
&      \leq              \frac{1}{2} \left\langle v_{k + 1} - v_{k}, v_{k + 1} \right\rangle + \left( 1 + \sqrt{\mu s} \right) \left\langle \nabla f(x_{k + 1}), x_{k + 1} - x_{k} \right\rangle   \\
&\mathrel{ \phantom{\leq} }                 -   \frac{1}{2} \left\langle ( 1 + \sqrt{\mu s} ) \sqrt{s} \nabla f(x_{k + 1}), 2 \sqrt{\mu} (x_{k + 1} - x^\star ) + v_{k + 1} + \sqrt{s} \nabla f(x_{k + 1})  \right\rangle    \\
& \mathrel{ \phantom{\leq} }                -    \frac{1}{4} \left\| v_{k + 1} - v_{k} \right\|^{2} - \frac{1}{4} \left\| \left( 1 + \sqrt{\mu s} \right) \sqrt{s} \nabla f(x_{k + 1})  \right\|^{2}                                      \\
&  \leq         -     \sqrt{\mu s} \left( \left\| v_{k + 1} \right\|^{2} + \left( 1 + \sqrt{\mu s} \right) \left\langle \nabla f(x_{k + 1}), x_{k + 1} - x^\star \right\rangle + \frac{s}{2} \left\| \nabla f(x_{k + 1}) \right\|^{2}  \right) \\
& \mathrel{ \phantom{\leq} }                -     \frac{1}{2} \left\langle \nabla f(x_{k + 1}) - \nabla f(x_{k}), x_{k + 1} - x_{k} \right\rangle  - \frac{s}{2} \left\| \nabla f(x_{k + 1})\right\|^{2} \\
                                                             &  \leq  - \frac{\sqrt{\mu s}}{4}  \mathcal{E}(k + 1).             
\end{align*}
Hence, the proof is complete.
\end{enumerate}

\subsection{Proof of Theorem~\ref{thm:heavy-ball-rate-fixed-s}}
\label{subsec: proof_polyak}

Here, we first describe and prove Theorem~\ref{thm:heavy-ball-rate-general-s} below. Then we complete the proof of Theorem~\ref{thm:heavy-ball-rate-fixed-s} by viewing it as a special case of Theorem~\ref{thm:heavy-ball-rate-general-s}.

\begin{thm}[Discretization of heavy-ball ODE --- General]\label{thm:heavy-ball-rate-general-s}
For any $f \in \mathcal{S}_{\mu, L}^{1}(\mathbb{R}^n)$, the following conclusions hold:
\begin{enumerate}
\item[(a)]
Taking $0 < s \leq \mu/(16L^2)$, the symplectic Euler scheme satisfies
\begin{align}
\label{eqn: conver_high_hb_sym}
f(x_k) - f(x^\star) \leq \left(  \frac{ \left(3 +  8\sqrt{\mu s} + 8 \mu s\right) sL}{(1 + \sqrt{\mu s})^{2}} + \frac{2\mu}{L} + \frac{1 + \sqrt{\mu s}}{2} \right)\frac{L \left\| x_0 - x^\star \right\|^{2} }{ \left(1 + \frac{\sqrt{\mu s}}{4} \right)^{k}}.                                       \end{align}

\item[(b)]
Taking $0 < s \leq \mu/(36L^2)$, the explicit Euler scheme satisfies
\begin{multline}
\label{eqn: conver_high_hb_ex}
f(x_k) - f(x^\star) \\ \leq  \left(  \frac{3sL}{(1 + \sqrt{\mu s})^{2}} + \frac{2\mu}{L} + \frac{1 + \sqrt{\mu s}}{2} \right)L \left\| x_0 - x^\star \right\|^{2}  \left(1 - \frac{\sqrt{\mu s}}{8} \right)^{k}.                                 
\end{multline}

\item[(c)] Taking $0 < s \leq 1/L $, the implicit Euler scheme satisfies
\begin{align}\label{eqn: conver_high_hb_im}
f(x_k) - f(x^\star) \leq  \left(  \frac{3sL}{(1 + \sqrt{\mu s})^{2}} + \frac{2\mu}{L} + \frac{1 + \sqrt{\mu s}}{2} \right) \frac{ L \left\| x_0 - x^\star \right\|^{2} }{\left(1 + \frac{\sqrt{\mu s}}{4} \right)^{k} }.
\end{align}
  
\end{enumerate}

\end{thm}

\paragraph{Proof of Theorem~\ref{thm:heavy-ball-rate-general-s}}
\begin{enumerate}

\item[(a)] The Lyapunov function is 
\[
\mathcal{E}(k) = \frac{1}{4} \left\| v_{k} \right\|^{2} + \frac{1}{4} \left\| 2 \sqrt{\mu} (x_{k} - x^\star) + v_{k} \right\|^{2} + (1 + \sqrt{\mu s}) \left( f(x_k) - f(x^\star) \right).
\]
With the Cauchy-Schwartz inequality
\[
\left\| 2 \sqrt{\mu} (x_k - x^\star) + v_{k} \right\|^{2} \leq 2 \left( 4 \mu \left\| x_k - x^\star \right\|^{2} + \left\| v_k \right\|^{2} \right),
\]
and the basic inequality for $f \in \mathcal{S}_{\mu, L}^{1}(\mathbb{R}^{n})$
\[
\left\{ \begin{aligned}
          & f(x_{k + 1}) - f(x_{k}) \leq \left\langle \nabla f(x_{k + 1}), x_{k + 1} - x_{k} \right\rangle \\
          & f(x^\star) \geq f(x_{k + 1}) + \left\langle \nabla f(x_{k + 1}), x^\star - x_{k + 1} \right\rangle + \frac{\mu}{2} \left\| x_{k + 1} - x^\star\right\|^{2},
         \end{aligned} \right. 
\]
then we can calculate the iterative difference
\begin{align*}
\lefteqn{\mathcal{E}(k + 1) - \mathcal{E}(k)} \\
 & =                \frac{1}{4} \left\langle v_{k + 1} - v_{k}, v_{k + 1} + v_{k} \right\rangle + (1 + \sqrt{\mu s}) \left( f(x_{k + 1}) - f(x_{k}) \right) \\
& \mathrel{ \phantom{=} }   +   \frac{1}{4} \left\langle 2\sqrt{\mu}(x_{k + 1} - x_{k}) + v_{k + 1} - v_{k},  2\sqrt{\mu}(x_{k + 1} + x_{k} - 2x^\star) + v_{k + 1} + v_{k} \right\rangle \\
& \leq           \frac{1}{2} \left\langle v_{k + 1} - v_{k}, v_{k + 1} \right\rangle + (1 + \sqrt{\mu s})  \left\langle \nabla f(x_{k + 1}), x_{k + 1} - x_{k} \right\rangle \\
& \mathrel{ \phantom{\leq} }    +   \frac{1}{2} \left\langle 2\sqrt{\mu}(x_{k + 1} - x_{k}) + v_{k + 1} - v_{k}, 2 \sqrt{\mu} (x_{k + 1} - x^\star) + v_{k + 1} \right\rangle \\
& \mathrel{ \phantom{\leq} }     -   \frac{1}{4} \left\| v_{k + 1} - v_{k} \right\|^{2} - \frac{1}{4} \left\| 2 \sqrt{\mu} (x_{k + 1} - x_{k}) + v_{k + 1} - v_{k} \right\|^{2} \\
& \leq       -   \sqrt{\mu s} \left[ \left\| v_{k + 1} \right\|^{2} + (1 + \sqrt{\mu s}) \left\langle \nabla f(x_{k + 1}), x_{k + 1} - x^\star \right\rangle \right] \\
& \leq       -   \sqrt{\mu s} \left[ \left\| v_{k + 1} \right\|^{2} + (1 + \sqrt{\mu s}) \left( f(x) - f(x^\star) \right) + \frac{\mu}{2} \left\| x_{k + 1} - x^\star \right\|^{2} \right] \\
 & \leq       -   \frac{\sqrt{\mu s}}{4} \mathcal{E}(k + 1).
\end{align*}
Hence, the proof is complete.


\item [(b)]
The Lyapunov function is 
\[
\mathcal{E}(k) = \frac{1}{4} \left\| v_k \right\|^{2} + \frac{1}{4} \left\| 2 \sqrt{\mu} (x_k - x^\star) + v_k \right\|^{2} + \left( 1 + \sqrt{\mu s} \right) \left( f(x_k) - f(x^\star) \right).
\]
With the Cauchy-Schwartz inequality
\[
\left\| 2 \sqrt{\mu} (x_k - x^\star) + v_k \right\|^{2} \leq 2 \left( 4 \mu \left\| x_k - x^\star \right\|^{2} + \left\| v_k \right\|^{2} \right),
\]
and the basic inequality for $f \in \mathcal{S}_{\mu, L}^{1}(\mathbb{R}^n)$
\[
\left\{\begin{aligned} 
         & f(x_{k + 1}) - f(x_{k}) \leq \left\langle \nabla f(x_{k}), x_{k + 1} - x_{k}  \right\rangle + \frac{L}{2} \left\| x_{k + 1} - x_{k} \right\|^{2}  \\
         & f(x^\star) \geq f(x_{k + 1}) + \left\langle \nabla f(x_{k + 1}), x^\star - x_{k + 1} \right\rangle + \frac{\mu}{2} \left\| x^\star - x_{k + 1} \right\|^{2},
        \end{aligned} \right.
\]
then we calculate the iterative difference
\begin{align*}
& \lefteqn{\mathcal{E}(k + 1) - \mathcal{E}(k)} \\
   & =            \frac{1}{4} \left\langle v_{k + 1} - v_{k}, v_{k + 1} + v_{k} \right\rangle + \left( 1 + \sqrt{\mu s} \right) \left( f(x_{k + 1}) - f(x_{k}) \right) \\
   &\mathrel{ \phantom{=} }   + \frac{1}{4} \left\langle 2 \sqrt{\mu} (x_{k + 1} - x_{k}) + v_{k + 1} - v_{k}, 2 \sqrt{\mu} (x_{k + 1} + x_{k} - 2x^\star) + v_{k + 1} + v_{k} \right\rangle \\
    & \leq        \frac{1}{2} \left\langle v_{k + 1} - v_{k}, v_{k} \right\rangle + \left( 1 + \sqrt{\mu s} \right) \left\langle \nabla f(x_{k}), x_{k + 1} - x_{k} \right\rangle \\
    & \mathrel{ \phantom{\leq} } + \frac{\left( 1 + \sqrt{\mu s} \right) L}{2} \left\| x_{k + 1} - x_{k} \right\|^{2}  \\
    & \mathrel{ \phantom{\leq} }  + \frac{1}{2} \left\langle 2 \sqrt{\mu} (x_{k + 1} - x_{k}) + v_{k + 1} - v_{k}, 2 \sqrt{\mu} (x_{k} - x^\star) + v_{k} \right\rangle \\
     & \mathrel{ \phantom{\leq} }  + \frac{1}{4} \left\| v_{k + 1} - v_{k} \right\|^{2} + \frac{1}{4} \left\| 2 \sqrt{\mu} (x_{k + 1} - x_{k}) + v_{k + 1} - v_{k} \right\|^{2} \\
     & \leq      - \sqrt{\mu s} \left( \left\| v_{k} \right\|^{2} + \left( 1 + \sqrt{\mu s} \right) \left\langle \nabla f(x_{k}), x_{k} - x^\star \right\rangle \right) \\
     & \mathrel{ \phantom{\leq} }    + \frac{ \left( 1 + \sqrt{\mu s} \right) Ls}{2} \left\| v_{k} \right\|^{2} + \frac{s}{4} \left\| 2\sqrt{\mu} v_{k} + \nabla f(x_{k}) \right\|^{2} + \frac{s}{4} \left\| \nabla f(x_{k}) \right\|^{2} \\
     & \leq      - \frac{\sqrt{\mu s}}{2} \left( \left\| v_{k } \right\|^{2} + f(x_{k }) - f(x^\star) + \frac{\mu}{2}\left\| x_{k} - x^\star \right\|^{2}  \right) \\  
     & \mathrel{ \phantom{\leq} }    - \frac{\sqrt{\mu s}}{2} \left( \left\| v_{k} \right\|^{2} +  \frac{\left( 1 + \sqrt{\mu s} \right)}{L}\left\| \nabla f(x_{k}) \right\|^2 \right) \\
     & \mathrel{ \phantom{\leq} }    + s \left( 2 \mu + \frac{L \left( 1 + \sqrt{\mu s} \right)}{2} \right) \left\| v_k \right\|^{2} + \frac{3s}{4} \left\| \nabla f(x_{k}) \right\|^{2}.                                                  
\end{align*}
Since $\mu \leq L$, then the step size $s \leq \mu/ (36L^{2})$ satisfies it. Hence, the proof is complete.


\item[(c)]
The Lyapunov function is constructed as
\[
\mathcal{E}(k) = \frac{1}{4} \left\| v_k \right\|^{2} + \frac{1}{4} \left\| 2 \sqrt{\mu} (x_{k + 1} - x^\star) + v_k \right\|^{2} + \left( 1 + \sqrt{\mu s} \right) \left( f(x_k) - f(x^\star) \right).
\]
With Cauchy-Schwartz inequality
\begin{align*}
\left\| 2 \sqrt{\mu} (x_{k + 1} - x^\star) + v_k \right\|^{2}   &   = \left\| 2 \sqrt{\mu} (x_k - x^\star) + (1 + 2 \sqrt{\mu s}) v_k \right\|^2 \\
                                                                                         &  \leq 2 \left( 4 \mu \left\| x_k - x^\star \right\|^{2} +  (1 + 2 \sqrt{\mu s})^2 \left\| v_k \right\|^{2} \right),
\end{align*}
and the basic inequality for $f \in \mathcal{S}_{\mu, L}^{1}(\mathbb{R}^n)$
\[
\left\{\begin{aligned} 
         & f(x_{k + 1}) - f(x_{k}) \leq \left\langle \nabla f(x_{k + 1}), x_{k + 1} - x_{k}  \right\rangle  - \frac{1}{2L} \left\| \nabla f(x_{k + 1}) - \nabla f(x_{k}) \right\|^{2} \\
         & f(x^\star) \geq f(x_{k + 1}) + \left\langle \nabla f(x_{k + 1}), x^\star - x_{k + 1} \right\rangle + \frac{\mu}{2} \left\| x^\star - x_{k + 1} \right\|^{2},
        \end{aligned} \right.
\]
then we calculate the iterative difference
\begin{align*}
\lefteqn{\mathcal{E}(k + 1) - \mathcal{E}(k)}   \\
 & =            \frac{1}{4} \left\langle v_{k + 1} - v_{k}, v_{k + 1} + v_{k} \right\rangle + \left( 1 + \sqrt{\mu s} \right) \left( f(x_{k + 1}) - f(x_{k}) \right) \\
 &  \mathrel{ \phantom{=} }  + \frac{1}{4} \left\langle 2 \sqrt{\mu} (x_{k + 2} - x_{k + 1}) + v_{k + 1} - v_{k}, 2 \sqrt{\mu} (x_{k + 2} + x_{k + 1} - 2x^\star) + v_{k + 1} + v_{k} \right\rangle \\
 &  \leq        \frac{1}{2} \left\langle v_{k + 1} - v_{k}, v_{k + 1} \right\rangle  - \frac{1}{4} \left\| v_{k + 1} - v_{k} \right\|^{2}  \\
 & \mathrel{ \phantom{\leq} }     + \left( 1 + \sqrt{\mu s} \right) \left\langle \nabla f(x_{k + 1}), x_{k + 1} - x_{k} \right\rangle  - \frac{\left( 1 + \sqrt{\mu s} \right)}{2L} \left\| \nabla f(x_{k + 1}) - \nabla f(x_{k}) \right\|^{2}  \\
& \mathrel{ \phantom{\leq} } + \frac{1}{2} \left\langle - \sqrt{s} \left( 1 + \sqrt{\mu s} \right) \nabla f(x_{k + 1}), 2 \sqrt{\mu} (x_{k + 2} - x^\star) + v_{k + 1} \right\rangle \\
& \mathrel{ \phantom{\leq} } - \frac{1}{4} \left\| \sqrt{s} \left( 1 + \sqrt{\mu s} \right) \nabla f(x_{k + 1}) \right\|^{2} \\
& \leq      - \sqrt{\mu s} \left( \left\| v_{k + 1} \right\|^{2} +  \left( 1 + \sqrt{\mu s} \right) \left\langle \nabla f(x_{k + 1}), x_{k + 1} - x^\star \right\rangle \right) \\
& \mathrel{ \phantom{\leq} }  - \frac{\sqrt{s}  \left( 1 + \sqrt{\mu s} \right) }{2} \left\langle \nabla f(x_{k + 1}), (1 + 2\sqrt{\mu s})v_{k + 1} - v_{k} \right\rangle  \\
& \mathrel{ \phantom{\leq} }    -  \frac{\left( 1 + \sqrt{\mu s} \right)}{2L} \left\| \nabla f(x_{k + 1}) - \nabla f(x_{k}) \right\|^{2}  - \frac{1}{4} \left\| v_{k + 1} - v_{k} + \sqrt{s} \nabla f(x_{k + 1}) \right\|^2    \\
& \leq      -  \sqrt{\mu s} \left[ \left\| v_{k + 1} \right\|^{2} + \frac{1}{4}  \left( 1 + \sqrt{\mu s} \right) \left( f(x_{k + 1}) - f(x^\star) \right) + \frac{\mu}{2} \left\| x_{k + 1} - x^\star \right\|^{2}  \right] \\
& \mathrel{ \phantom{\leq} }    - \frac{ \left( 1 + \sqrt{\mu s} \right)}{4} \left[ 3\sqrt{\mu s} \left( f(x_{k + 1}) - f(x^\star) \right) - 2 s \left\| \nabla f(x_{k + 1}) \right\|^{2} \right]                                          
\end{align*}
Since $\mu \leq L$, then the step size $s \leq \mu/ (16L^{2})$ satisfies it. Hence, the proof is complete with some basic calculations.

\end{enumerate}

\section{Technical Analysis and Proofs for Section~\ref{sec:high-resolution-ode}}
\label{sec: high_nag-c}
\subsection{Technical details for numerical scheme of ODE~\eqref{eqn: nag-c_phase}}
\label{subsec: phase-space_numercial}

\paragraph{Proof of Theorem~\ref{thm:nag-c-all-phase-space} (b)}
The Lyapunov function is constructed as
\begin{multline*}
\mathcal{E}(k) = s\left( k + 2 \right) \left( k + 3 \right) \left( f(x_{k}) - f(x^\star) \right) + \frac{1}{2} \left\| 2 (x_{k} - x^\star) + (k + 1) \sqrt{s} \left( v_{k} + \sqrt{s} \nabla f(x_{k}) \right) \right\|^{2}.
\end{multline*}
With the basic inequality for $f \in \mathcal{F}_{L}^{1}(\mathbb{R}^n)$ 
\[
\left\{ \begin{aligned}
         & f(x_{k + 1}) - f(x_{k})   \leq \left\langle \nabla f(x_{k + 1}), x_{k + 1} - x_{k} \right\rangle \\
         & f(x_{k + 1}) - f(x^\star) \leq \left\langle \nabla f(x_{k + 1}), x_{k + 1} - x^\star \right\rangle,  
         \end{aligned}\right.
\]
we can calculate the iterative difference as
\begin{align*}
         &  \lefteqn{\mathcal{E}(k + 1) - \mathcal{E}(k)} \\
       & = s\left( k + 2\right) \left( k + 3 \right) \left( f(x_{k + 1}) - f(x_{k}) \right) + s \left( 2k + 6 \right) \left( f(x_{k + 1}) - f(x^\star) \right) \\
         &\mathrel{ \phantom{=} }   \begin{aligned} 
         + \frac{1}{2}\left\langle 2(x_{k + 1} - x_{k})\right. & - \sqrt{s}(k + 1) \left( v_{k} + \sqrt{s} \nabla f(x_{k}) \right) \\ & + \sqrt{s}(k + 2) \left( v_{k + 1} + \sqrt{s} \nabla f(x_{k + 1})\right) ,   
            \end{aligned}    \\    
              & \mathrel{ \phantom{=+\frac{1}{2}\langle}}   \begin{aligned} 2(x_{k + 1} + x_{k} - 2x^\star) & + \sqrt{s}(k + 1) \left( v_{k} + \sqrt{s} \nabla f(x_{k}) \right)\\
               &+\left.  \sqrt{s}(k + 2) \left( v_{k + 1} + \sqrt{s} \nabla f(x_{k + 1})\right) \right\rangle 
         \end{aligned}
         \\
         &=  s \left( k + 2 \right) \left( k + 3 \right) \left( f(x_{k + 1}) - f(x_{k}) \right) + s\left( 2k + 6 \right) \left( f(x_{k + 1}) - f(x^\star) \right) \\
         &\mathrel{ \phantom{=} } \begin{aligned} 
                                                  - \langle & s \left( k + 3\right) \nabla f(x_{k + 1}), \\
                                                                & 2 (x_{k + 1} - x^\star) + \sqrt{s} (k + 2) \left( v_{k + 1} + \left.\sqrt{s} \nabla f(x_{k + 1}) \right)\right\rangle  
                                                  \end{aligned} \\
         & \mathrel{ \phantom{=} }  - \frac{s^{2}}{2} \left( k + 3 \right)^{2}  \left\| \nabla f(x_{k + 1}) \right\|^{2} \\
         & \leq - \frac{s^{2}}{2} \left( k + 3 \right) \left( 3k + 7 \right) \left\| \nabla f(x_{k + 1}) \right\|^{2}.        
\end{align*}

Hence, the proof is complete with some basic calculations.

\paragraph{Technical analysis of explicit Euler of ODE~\eqref{eqn: nag-c_phase}}

The Lyapunov function is 
\[
\mathcal{E}(k) = s\left( k - 2 \right) \left( k +1 \right) \left( f(x_{k}) - f(x^\star) \right) + \frac{1}{2} \left\| 2 (x_{k} - x^\star) + (k - 1) \sqrt{s} \left( v_{k} + \sqrt{s} \nabla f(x_{k}) \right) \right\|^{2}.
\]
Then we calculate the iterative difference as
\begin{align*}
       & \lefteqn{ \mathcal{E}(k + 1) - \mathcal{E}(k) } \\
       & = s \left( k - 1 \right) \left( k + 2 \right) \left( f(x_{k + 1}) - f(x_{k}) \right) +  2sk \left( f(x_{k }) - f(x^\star) \right) \\
       & \mathrel{ \phantom{=} } 
       \begin{aligned} + \frac{1}{2}\langle 2(x_{k + 1} - x_{k})  & + \sqrt{s}k \left( v_{k + 1} + \sqrt{s} \nabla f(x_{k + 1})\right) \\ 
                                                                                              &- \sqrt{s}(k - 1) \left( v_{k} + \sqrt{s} \nabla f(x_{k}) \right),
       \end{aligned} \\
       &\mathrel{ \phantom{=}}\mathrel{ \phantom{+\frac{1}{2}\langle} } 
       \begin{aligned} 2(x_{k + 1} + x_{k} - 2x^\star) &+ \sqrt{s}k \left( v_{k + 1} + \sqrt{s} \nabla f(x_{k + 1})\right) \\& + \left.\sqrt{s}(k - 1) \left( v_{k} + \sqrt{s} \nabla f(x_{k}) \right) \right\rangle
       \end{aligned} \\
       & = s \left( k - 1 \right) \left( k + 2 \right) \left( f(x_{k+1}) - f(x_{k}) \right) + 2sk \left( f(x_{k }) - f(x^\star) \right) \\
       & \mathrel{ \phantom{=} } - \left\langle s \left( k + 2 \right) \nabla f(x_{k }), 2 (x_{k} - x^\star) + \sqrt{s} (k - 1) \left( v_{k} + \sqrt{s} \nabla f(x_{k }) \right) \right\rangle   \\
       & \mathrel{ \phantom{=} }  + \frac{s^{2}}{2} \left( k + 2 \right)^{2}  \left\| \nabla f(x_{k}) \right\|^{2}. 
\end{align*}         
\begin{itemize}
\item If we take the following basic inequality for $f \in \mathcal{F}_{L}^{1}(\mathbb{R}^{n})$
         \[
         \left\{ \begin{aligned}
                  & f(x_{k + 1}) - f(x_{k}) \leq \left\langle \nabla f(x_{k}), x_{k + 1} - x_{k} \right\rangle + \frac{L}{2} \left\| x_{k + 1} - x_{k} \right\|^{2} \\
                  & f(x_{k}) - f(x^\star) \leq \left\langle \nabla f(x_{k}), x_{k} - x^\star \right\rangle,
                 \end{aligned} \right.
         \]
         we can obtain the following estimate
         \begin{multline*}
         \mathcal{E}(k + 1) - \mathcal{E}(k) \\ \leq \frac{Ls^{2}}{2} \left( k - 1 \right)\left( k + 2 \right) \left\| v_{k} \right\|^{2} - 4s \left( f(x_{k}) - f(x^\star) \right) - \frac{s^{2}}{2} \left( k + 2 \right) \left( k - 4 \right) \left\| \nabla f(x_{k}) \right\|^{2},
         \end{multline*}
         which cannot guarantee the right-hand side of the inequality non-positive.
\item If we take the following basic inequality for $f \in \mathcal{F}_{L}^{1}(\mathbb{R}^{n})$
         \[
         \left\{ \begin{aligned}
                  & f(x_{k + 1}) - f(x_{k}) \leq \left\langle \nabla f(x_{k + 1}), x_{k + 1} - x_{k} \right\rangle - \frac{1}{2L} \left\| \nabla f(x_{k + 1}) - \nabla f(x_{k}) \right\|^{2} \\
                  & f(x_{k}) - f(x^\star) \leq \left\langle \nabla f(x_{k}), x_{k} - x^\star \right\rangle,
                 \end{aligned} \right.
         \]
         we can obtain the following estimate
         \begin{align*}
                &\lefteqn{ \mathcal{E}(k + 1) - \mathcal{E}(k)}            \\  
         &   \begin{aligned} \leq \left.\frac{Ls\left( k - 1 \right)\left( k + 2 \right)}{2}  \right( &\left\langle \nabla f(x_{k + 1}) - \nabla f(x_{k}), x_{k + 1} - x_{k} \right\rangle \\ & \quad - \left.\frac{1}{2L} \left\| \nabla f(x_{k + 1}) - \nabla f(x_{k}) \right\|^{2} t \right) 
         \end{aligned} \\
                &  \mathrel{ \phantom{\leq} }  - 4s \left( f(x_{k}) - f(x^\star) \right) - \frac{s^{2}}{2} \left( k + 2 \right) \left( k - 4 \right)  \left\| \nabla f(x_{k}) \right\|^{2},
         \end{align*}
         which cannot guarantee the right-hand side of the inequality non-positive.
\end{itemize}

\subsection{Technical details for standard numerical schemes}
\label{subsec: standard_numerical}

Standard Euler discretization of ODE~\eqref{eqn: low_nag-c_phase}, with initial $x _{0}$ and $v_{0} = - \sqrt{s} \nabla f(x_{0})$, are shown as below.
\textbf{Euler scheme of~\eqref{eqn: low_nag-c_phase}: (S), (E) and (I) respectively}
 \begin{align*}
&\textbf{(S)} && \left\{ \begin{aligned}
           & x_{k + 1} - x_{k} = \sqrt{s} v_{k}                                                          \\
           & v_{k + 1} - v_{k} = - \frac{3 v_{k + 1}}{k + 1} - \sqrt{s} \left( \nabla f(x_{k + 1}) - \nabla f(x_{k}) \right)- \sqrt{s}\left( \frac{2 k + 5}{2  k + 2}\right) \nabla f(x_{k + 1}).  
          \end{aligned} \right. \\
&\textbf{(E)} && \left\{ \begin{aligned}
           & x_{k + 1} - x_{k} = \sqrt{s} v_{k}                                                          \\
           & v_{k + 1} - v_{k} = - \frac{3v_{k}}{k} - \sqrt{s} \left( \nabla f(x_{k + 1}) - \nabla f(x_{k}) \right)- \sqrt{s}\left( \frac{2k + 3}{2k}\right) \nabla f(x_{k}). 
          \end{aligned} \right. \\
&\textbf{(I)} && \left\{ \begin{aligned}
           & x_{k + 1} - x_{k} = \sqrt{s} v_{k + 1}                                                          \\
           & v_{k + 1} - v_{k} = - \frac{3v_{k + 1}}{k + 1} - \sqrt{s} \left( \nabla f(x_{k + 1}) - \nabla f(x_{k}) \right)- \sqrt{s}\left(\frac{2 k + 5}{2 k + 2}\right) \nabla f(x_{k + 1}).  
          \end{aligned} \right. 
\end{align*}

\paragraph{Technical analysis of symplectic scheme of ODE~\eqref{eqn: low_nag-c_phase}}

The Lyapunov function is 
\[
\mathcal{E}(k) = s\left( k + 1\right) \left( k + \frac{3}{2} \right) \left( f(x_{k}) - f(x^\star) \right) + \frac{1}{2} \left\| 2 (x_{k + 1} - x^\star) + (k + 1) \sqrt{s} \left( v_{k} + \sqrt{s} \nabla f(x_{k}) \right) \right\|^{2}.
\]
Then we calculate the iterative difference as
\begin{align*}
          \lefteqn{ \mathcal{E}(k + 1) - \mathcal{E}(k)} \\
 & =        s\left( k + 1\right) \left( k + \frac{3}{2} \right) \left( f(x_{k + 1}) - f(x_{k}) \right) + s \left( 2k + \frac{7}{2} \right) \left( f(x_{k + 1}) - f(x^\star) \right) \\
         & \mathrel{ \phantom{=} } \begin{aligned} + \frac{1}{2}\left\langle 2(x_{k + 2} - x_{k + 1})  \right. & - \sqrt{s}(k + 1) \left( v_{k} + \sqrt{s} \nabla f(x_{k}) \right) \\
                                                                                                                                                               & + \sqrt{s}(k + 2) \left( v_{k + 1} + \sqrt{s} \nabla f(x_{k + 1}) \right) 
                          \end{aligned} \\
         & \mathrel{ \phantom{=} }  \begin{aligned}\mathrel{ \phantom{ +\frac{1}{2}\langle} } 2(x_{k + 2} + x_{k + 1} - 2x^\star)& +  \sqrt{s}(k + 1) \left( v_{k} + \sqrt{s} \nabla f(x_{k}) \right) \\ &+  \left. \sqrt{s}(k + 2) \left( v_{k + 1} + \sqrt{s} \nabla f(x_{k + 1})\right) \right\rangle   
         \end{aligned} \\
& =       s \left( k + 1 \right) \left( k + \frac{3}{2} \right) \left( f(x_{k + 1}) - f(x_{k}) \right) + s\left( 2k + \frac{7}{2} \right) \left( f(x_{k + 1}) - f(x^\star) \right) \\
         & \mathrel{ \phantom{=} } - \left\langle s \left( k + \frac{3}{2} \right) \nabla f(x_{k + 1}), 2 (x_{k + 2} - x^\star) + \sqrt{s} (k + 2) \left( v_{k + 1} + \sqrt{s} \nabla f(x_{k + 1}) \right) \right\rangle   \\
         & \mathrel{ \phantom{=} } - \frac{s^{2}}{2} \left( k + \frac{3}{2} \right)^{2}  \left\| \nabla f(x_{k + 1}) \right\|^{2}. 
\end{align*}

Now we hope to utilize the basic inequality for $f \in \mathcal{F}_{L}^{1}(\mathbb{R}^n)$ to make the right side of equality no more than zero. Taking the following inequalities
\[
\left\{ \begin{aligned}
         & f(x_{k + 1}) - f(x_{k})   \leq \left\langle \nabla f(x_{k + 1}), x_{k + 1} - x_{k} \right\rangle  -  \frac{1}{2L} \left\| \nabla f(x_{k + 1}) - \nabla f(x_{k}) \right\|^{2} \\
         & f(x_{k + 1}) - f(x^\star) \leq \left\langle \nabla f(x_{k + 1}), x_{k + 1} - x^\star \right\rangle,  
         \end{aligned}\right.
\]
we can obtain the iterative difference is
\begin{align*}
& \lefteqn{ \mathcal{E}(k + 1) - \mathcal{E}(k)} \\
   & \leq     \frac{s}{2}  \left( f(x_{k + 1}) - f(x^\star) \right)  - \frac{s^{2}}{2} \left( k + \frac{3}{2}\right) \left( 3k + \frac{7}{2} \right) \left\| \nabla f(x_{k + 1}) \right\|^{2} \\  
                                                                     & \mathrel{ \phantom{\leq}}  - \frac{s}{2L} \left(k + 1 \right) \left( k + \frac{3}{2} \right) \left\| \nabla f(x_{k + 1}) - \nabla f(x_{k}) \right\|^{2} \\
                                                                     & \mathrel{ \phantom{\leq}}+ s^{2} \left(k + 1 \right) \left( k + \frac{3}{2} \right) \left\langle \nabla f(x_{k + 1}), \nabla f(x_{k + 1}) - \nabla f(x_{k})  \right\rangle \\
                                                                     &\mathrel{ \phantom{\leq}} + s^{2} \left( k + \frac{3}{2} \right) \left( k + \frac{5}{2} \right) \left\| \nabla f(x_{k + 1}) \right\|^{2} \\
                                                            & \leq    \frac{s}{2}  \left( f(x_{k + 1}) - f(x^\star) \right)  - \frac{s^{2}}{2} \left( k + \frac{3}{2} \right) \left( k - \frac{3}{2} - Ls (k + 1) \right) \left\| \nabla f(x_{k + 1}) \right\|^{2}.        
\end{align*}

Since there exists a non-negative term,~$\frac{s}{2} \left( f(x_{k + 1}) - f(x^\star) \right)$, we cannot guarantee the right-hand side of inequality is non-positive. Hence, the convergence cannot be proved by the above description.

\paragraph{Technical analysis of explicit scheme of ODE~\eqref{eqn: low_nag-c_phase}}

The Lyapunov function is 
\[
\mathcal{E}(k) = s\left( k - 2 \right) \left( k - \frac{1}{2} \right) \left( f(x_{k}) - f(x^\star) \right) + \frac{1}{2} \left\| 2 (x_{k} - x^\star) + (k - 1) \sqrt{s} \left( v_{k} + \sqrt{s} \nabla f(x_{k}) \right) \right\|^{2}.
\]
Then we calculate the iterative difference as
\begin{align*}
         &\lefteqn{  \mathcal{E}(k + 1) - \mathcal{E}(k)} \\
         & = s \left( k - 1 \right) \left( k + \frac{1}{2} \right) \left( f(x_{k + 1}) - f(x_{k}) \right) + s \left( 2k - \frac{3}{2} \right) \left( f(x_{k }) - f(x^\star) \right) \\
         &  \mathrel{ \phantom{=} } \begin{aligned}+ \frac{1}{2}\left\langle 2(x_{k + 1} - x_{k}) \right. &+ \sqrt{s}k \left( v_{k + 1} + \sqrt{s} \nabla f(x_{k + 1})\right) \\&- \sqrt{s}(k - 1) \left( v_{k} + \sqrt{s} \nabla f(x_{k}) \right),  \end{aligned} \\
         &  \mathrel{ \phantom{=}}  \begin{aligned} \mathrel{ \phantom{+ \frac{1}{2}\langle }}  2(x_{k + 1} + x_{k} - 2x^\star) &+ \sqrt{s}k \left( v_{k + 1}  + \sqrt{s} \nabla f(x_{k + 1})\right) \\
                                                                                                                                    & + \left.\sqrt{s}(k - 1) \left( v_{k} + \sqrt{s} \nabla f(x_{k}) \right) \right\rangle \end{aligned}\\
       & = s \left( k - 1 \right) \left( k + \frac{1}{2} \right) \left( f(x_{k+1}) - f(x_{k}) \right) + s\left( 2k - \frac{3}{2} \right) \left( f(x_{k }) - f(x^\star) \right) \\
         & \mathrel{ \phantom{=} } - \left\langle s \left( k + \frac{1}{2} \right) \nabla f(x_{k }), 2 (x_{k} - x^\star) + \sqrt{s} (k - 1) \left( v_{k} + \sqrt{s} \nabla f(x_{k }) \right) \right\rangle   \\
         & \mathrel{ \phantom{=} } + \frac{s^{2}}{2} \left( k + \frac{1}{2} \right)^{2}  \left\| \nabla f(x_{k}) \right\|^{2}. 
\end{align*}

\begin{itemize}
\item If we take the following basic inequality for $f \in \mathcal{F}_{L}^{1}(\mathbb{R}^{n})$
         \[
         \left\{ \begin{aligned}
                  & f(x_{k + 1}) - f(x_{k}) \leq \left\langle \nabla f(x_{k}), x_{k + 1} - x_{k} \right\rangle + \frac{L}{2} \left\| x_{k + 1} - x_{k} \right\|^{2} \\
                  & f(x_{k}) - f(x^\star) \leq \left\langle \nabla f(x_{k}), x_{k} - x^\star \right\rangle,
                 \end{aligned} \right.
         \]
         we can obtain the following estimate
         \begin{multline*}
         \mathcal{E}(k + 1) - \mathcal{E}(k) \leq \frac{Ls^{2}}{2} \left( k - 1 \right)\left( k + \frac{1}{2} \right) \left\| v_{k} \right\|^{2} \\ - \frac{5s}{2} \left( f(x_{k}) - f(x^\star) \right) - \frac{s^{2}}{2} \left( k + \frac{1}{2} \right) \left( k - \frac{5}{2} \right) \left\| \nabla f(x_{k}) \right\|^{2},
         \end{multline*}
         which cannot guarantee the right-hand side of the inequality non-positive.
\item If we take the following basic inequality for $f \in \mathcal{F}_{L}^{1}(\mathbb{R}^{n})$
         \[
         \left\{ \begin{aligned}
                  & f(x_{k + 1}) - f(x_{k}) \leq \left\langle \nabla f(x_{k + 1}), x_{k + 1} - x_{k} \right\rangle - \frac{1}{2L} \left\| \nabla f(x_{k + 1}) - \nabla f(x_{k}) \right\|^{2} \\
                  & f(x_{k}) - f(x^\star) \leq \left\langle \nabla f(x_{k}), x_{k} - x^\star \right\rangle,
                 \end{aligned} \right.
         \]
         we can obtain the following estimate
         \begin{align*}
                & \lefteqn{\mathcal{E}(k + 1) - \mathcal{E}(k)}            \\  
                &  \begin{aligned} \leq \left.\frac{Ls(k - 1)(2k + 1)}{4} \right( & \left\langle \nabla f(x_{k + 1}) - \nabla f(x_{k}), x_{k + 1} - x_{k} \right\rangle \\ & - \left.\frac{1}{2L} \left\| \nabla f(x_{k + 1}) - \nabla f(x_{k}) \right\|^{2} \right) \end{aligned}\\
                &\mathrel{ \phantom{\leq} } - \frac{5s}{2} \left( f(x_{k}) - f(x^\star) \right) - \frac{s^{2}}{2} \left( k + \frac{1}{2} \right) \left( k - \frac{5}{2} \right) \left\| \nabla f(x_{k}) \right\|^{2},
         \end{align*}
         which cannot guarantee the right-hand side of the inequality non-positive.
\end{itemize}

\paragraph{Technical analysis of implicit scheme of ODE~\eqref{eqn: low_nag-c_phase}}

The Lyapunov function is 
\begin{multline*}
\mathcal{E}(k) = s\left( k + 2 \right) \left( k + \frac{3}{2} \right) \left( f(x_{k}) - f(x^\star) \right) \\ + \frac{1}{2} \left\| 2 (x_{k} - x^\star) + (k + 1) \sqrt{s} \left( v_{k} + \sqrt{s} \nabla f(x_{k}) \right) \right\|^{2}.
\end{multline*}
Then we can calculate the iterative difference as
\begin{align*}
         &  \lefteqn{\mathcal{E}(k + 1) - \mathcal{E}(k)} \\
         &  =s\left( k + 2\right) \left( k + \frac{3}{2} \right) \left( f(x_{k + 1}) - f(x_{k}) \right) + s \left( 2k + \frac{9}{2} \right) \left( f(x_{k + 1}) - f(x^\star) \right) \\
           & \mathrel{ \phantom{=} } \begin{aligned}+ \frac{1}{2}\left\langle 2(x_{k + 1} - x_{k}) \right.  & -  \sqrt{s}(k + 1) \left( v_{k} + \sqrt{s} \nabla f(x_{k}) \right)  \\& + \sqrt{s}(k + 2) \left( v_{k + 1} + \sqrt{s} \nabla f(x_{k + 1})\right),  \end{aligned}\\
         & \mathrel{ \phantom{=} } \begin{aligned}  \mathrel{ \phantom{+\frac{1}{2}\langle} } 2(x_{k + 1} + x_{k} - 2x^\star) &+  \sqrt{s}(k + 1) \left( v_{k} + \sqrt{s} \nabla f(x_{k}) \right)\\ &+  \left. \sqrt{s}(k + 2) \left( v_{k + 1} + \sqrt{s} \nabla f(x_{k + 1})\right)   \right\rangle \end{aligned} \\
      & = s \left( k + 2 \right) \left( k + \frac{3}{2} \right) \left( f(x_{k + 1}) - f(x_{k}) \right) + s\left( 2k + \frac{9}{2} \right) \left( f(x_{k + 1}) - f(x^\star) \right) \\
         &\mathrel{ \phantom{=} }  - \left\langle s \left( k + \frac{3}{2} \right) \nabla f(x_{k + 1}), 2 (x_{k + 1} - x^\star) + \sqrt{s} (k + 2) \left( v_{k + 1} + \sqrt{s} \nabla f(x_{k + 1}) \right) \right\rangle   \\
         &\mathrel{ \phantom{=} }  - \frac{s^{2}}{2} \left( k + \frac{3}{2} \right)^{2}  \left\| \nabla f(x_{k + 1}) \right\|^{2}. 
\end{align*}

Now we hope to utlize the basic inequality for $f \in \mathcal{F}_{L}^{1}(\mathbb{R}^n)$ to make the right side of equality no more than zero. Taking the following inequalities
\[
\left\{ \begin{aligned}
         & f(x_{k + 1}) - f(x_{k})   \leq \left\langle \nabla f(x_{k + 1}), x_{k + 1} - x_{k} \right\rangle \\
         & f(x_{k + 1}) - f(x^\star) \leq \left\langle \nabla f(x_{k + 1}), x_{k + 1} - x^\star \right\rangle,  
         \end{aligned}\right.
\]
we can obtain 
\[
  \mathcal{E}(k + 1) - \mathcal{E}(k)  \leq     \frac{3s}{2}  \left( f(x_{k + 1}) - f(x^\star) \right)  - \frac{s^{2}}{2} \left( k + \frac{3}{2}\right) \left( 3k + \frac{7}{2} \right) \left\| \nabla f(x_{k + 1}) \right\|^{2}.   
\]

Although the negative term concludes the multiplier $k^{2}$, we cannot guarantee the right-hand side non-positive

\section{Low-Resolution ODEs}
\label{sec: low_res}

\subsection{Low-resolution ODE for strongly convex functions}

\label{subsec: low_sc}

In this subsection, we discuss the numerical discretization of~\eqref{eqn: low_hb_NAG-SC}. 
We rewrite this ODE in a phase-space representation  
\begin{equation}
\label{eqn: low_hb_NAG-SC_phase}
\left\{ \begin{aligned}
          & \dot{X} = V                                               \\
          & \dot{V} = - 2 \sqrt{\mu} V - \nabla f(X)  
         \end{aligned} \right. ,
\end{equation}
with $X(0) = x_0$ and $V(0) = 0$. We have the following theorem:
\begin{thm}
\label{thm: low_con_strongly}
 Let $f \in \mathcal{S}_{\mu, L}^{1}(\mathbb{R}^n)$. The solution $X = X(t)$ to low-resolution ODE~\eqref{eqn: low_hb_NAG-SC} satisfies 
\begin{align}
\label{eqn: conver_sc_low1}
f(X) - f(x^\star) \leq \frac{ 3L \left\| x_0 - x^\star \right\|^2}{2} \ee^{- \frac{\sqrt{\mu} t}{4}}.               \end{align}     
\end{thm}  
 
 \begin{proof}
 The Lyapunov function is 
\[
\mathcal{E} = \frac{1}{4} \| \dot{X} \|^{2} + \frac{1}{4} \| 2 \sqrt{\mu} (X - x^{\star}) + \dot{X} \|^{2} + f(X) - f(x^{\star}).
\]
Using the Cauchy-Schwartz inequality 
\[
   \| 2\sqrt{\mu} ( X - x^{\star} ) + \dot{X} \|^{2} \leq 2 \left( 4 \mu \| X - x^{\star} \|^{2} + \| \dot{X} \|^{2} \right),
\]
and the basic inequality for  $f \in \mathcal{S}_{\mu, L}^{1}(\mathbb{R}^n)$ 
\[
f(x^{\star}) \geq f(X) + \left\langle \nabla f(X), x^{\star} - x \right\rangle + \frac{\mu}{2} \left\| X - x^{\star} \right\|^{2},
\]
we calculate the time derivative 
\begin{align*}
\frac{d \mathcal{E}}{dt}     & = \frac{1}{2}\left\langle \dot{X}, - 2 \sqrt{\mu} \dot{X} - \nabla f(X) \right\rangle + \frac{1}{2} \left\langle 2\sqrt{\mu} \left( X - x^{\star} \right) + \dot{X}, - \nabla f(X) \right\rangle + \left\langle \nabla f(X), \dot{X} \right\rangle \\
                                         & = - \sqrt{\mu} \left( \| \dot{X} \|^{2} + \left\langle \nabla f(X), X - x^{\star} \right\rangle \right) \\
                                         & \leq - \sqrt{\mu} \left( \| \dot{X} \|^{2} + f(X) - f(x^{\star}) + \frac{\mu}{2} \left\| X - x^{\star} \right\|^{2} \right) \\                                    
                                         & \leq - \frac{\sqrt{\mu}}{4} \mathcal{E}.
\end{align*}
Hence, the proof is complete.
\end{proof}
 
 We now analyze the standard Euler discretization of the low-resolution ODE~\eqref{eqn: low_hb_NAG-SC}. All of the following  three Euler schemes take the same initial $x_{0}$ and $v_{0} = 0$.

\textbf{Euler Scheme of ODE~\eqref{eqn: low_hb_NAG-SC}: (S), (E) and (I) respectively}
\begin{align*}
&\textbf{(S)} && 
 \left\{ \begin{aligned}
          & x_{k + 1} - x_{k} = \sqrt{s} v_{k}                                                          \\
          & v_{k + 1} - v_{k} = - 2 \sqrt{\mu s} v_{k + 1} - \sqrt{s} \nabla f(x_{k + 1}).  
         \end{aligned} \right. \\
&\textbf{(E)} &&
\left\{ \begin{aligned}
          & x_{k + 1} - x_{k} = \sqrt{s} v_{k}                                                          \\
          & v_{k + 1} - v_{k} = - 2 \sqrt{\mu s} v_{k} - \sqrt{s} \nabla f(x_{k}).  
         \end{aligned} \right. \\
&\textbf{(I)} &&
\left\{ \begin{aligned}
          & x_{k + 1} - x_{k} = \sqrt{s} v_{k + 1}                                                          \\
          & v_{k + 1} - v_{k} = - 2 \sqrt{\mu s} v_{k + 1} - \sqrt{s} \nabla f(x_{k + 1}).  
         \end{aligned} \right. 
\end{align*}

\begin{thm}[Discretization of Low-Resolution ODE --- General]\label{thm: low_nag-sc-three}
For any $f \in \mathcal{S}_{\mu, L}^{1}(\mathbb{R}^n)$, the following conclusions hold:
\begin{enumerate}
\item[(a)] Taking $0 < s \leq \mu/(16L^2)$, the symplectic Euler scheme satisfies
\begin{align}
\label{eqn: conver_sc_low_sym}
f(x_k) - f(x^\star) \leq \frac{3L \left\| x_0 - x^\star \right\|^{2} }{2\left(1 + \frac{\sqrt{\mu s}}{4} \right)^{k}}.                                       
 \end{align}

\item[(b)] Taking$0 < s \leq \mu/(25L^2)$, the explicit Euler scheme satisfies
\begin{align}
\label{eqn: conver_sc_low_ex}
f(x_k) - f(x^\star) \leq \frac{3L \left\| x_0 - x^\star \right\|^{2} }{2}   \left(1 - \frac{\sqrt{\mu s}}{8} \right)^{k}.                                    
 \end{align}

\item[(c)] Taking $0 < s \leq 1/L$, the implicit Euler scheme satisfies
\begin{align}
\label{eqn: conver_sc_low_im}
f(x_k) - f(x^\star) \leq \frac{3L \left\| x_0 - x^\star \right\|^{2} }{2\left(1 + \frac{\sqrt{\mu s}}{4} \right)^{k} }.                                     
 \end{align}
\end{enumerate}
  
\end{thm}

\begin{proof}
\begin{enumerate}
\item [(a)]
The Lyapunov function is 
\[
\mathcal{E}(k) = \frac{1}{4} \left\| v_k \right\|^{2} + \frac{1}{4} \left\| 2 \sqrt{\mu} (x_k - x^\star) + v_k \right\|^{2} + f(x_k) - f(x^\star).
\]
With the Cauchy-Schwartz inequality
\[
\left\| 2 \sqrt{\mu} (x_k - x^\star) + v_k \right\|^{2} \leq 2 \left( 4 \mu \left\| x_k - x^\star \right\|^{2} + \left\| v_k \right\|^{2} \right),
\]
and the basic inequality for $f \in \mathcal{S}_{\mu, L}^{1}(\mathbb{R}^n)$
\[
\left\{\begin{aligned} 
         & f(x_{k + 1}) - f(x_{k}) \leq \left\langle \nabla f(x_{k + 1}), x_{k + 1} - x_{k}  \right\rangle   \\
         & f(x^\star) \geq f(x_{k + 1}) + \left\langle \nabla f(x_{k + 1}), x^\star - x_{k + 1} \right\rangle + \frac{\mu}{2} \left\| x_{k + 1} - x^\star \right\|^{2},
        \end{aligned} \right.
\]
we calculate the iterave difference
\begin{align*}
&\lefteqn{\mathcal{E}(k + 1) - \mathcal{E}(k)}   \\ & =            \frac{1}{4} \left\langle v_{k + 1} - v_{k}, v_{k + 1} + v_{k} \right\rangle + f(x_{k + 1}) - f(x_{k}) \\
                                                        & \mathrel{ \phantom{=} }   + \frac{1}{4} \left\langle 2 \sqrt{\mu} (x_{k + 1} - x_{k}) + v_{k + 1} - v_{k}, 2 \sqrt{\mu} (x_{k + 1} + x_{k} - 2x^\star) + v_{k + 1} + v_{k} \right\rangle \\
                                                        & \leq        \frac{1}{2} \left\langle v_{k + 1} - v_{k}, v_{k + 1} \right\rangle + \left\langle \nabla f(x_{k + 1}), x_{k + 1} - x_{k} \right\rangle \\
                                                        & \mathrel{ \phantom{\leq} }   + \frac{1}{2} \left\langle 2 \sqrt{\mu} (x_{k + 1} - x_{k}) + v_{k + 1} - v_{k}, 2 \sqrt{\mu} (x_{k + 1} - x^\star) + v_{k + 1} \right\rangle \\
                                                        &\mathrel{ \phantom{\leq} }    - \frac{1}{4} \left\| v_{k + 1} - v_{k} \right\|^{2} - \frac{1}{4} \left\| 2 \sqrt{\mu} (x_{k + 1} - x_{k}) + v_{k + 1} - v_{k} \right\|^{2} \\
                                                        & \leq      - \sqrt{\mu s} \left( \left\| v_{k + 1} \right\|^{2} + \left\langle \nabla f(x_{k + 1}), x_{k + 1} - x^\star \right\rangle \right) \\
                                                        & \leq      - \sqrt{\mu s} \left( \left\| v_{k + 1} \right\|^{2} + f(x_{k + 1}) - f(x^\star) + \frac{\mu}{2}\left\| x_{k + 1} - x^\star \right\|^{2}  \right) \\  
                                                        & \leq      - \frac{\sqrt{\mu s}}{4} \mathcal{E}(k + 1).                                                    
\end{align*}
Hence, the proof is complete.


\item [(b)] The Lyapunov function is 
\[
\mathcal{E}(k) = \frac{1}{4} \left\| v_k \right\|^{2} + \frac{1}{4} \left\| 2 \sqrt{\mu} (x_k - x^\star) + v_k \right\|^{2} + f(x_k) - f(x^\star).
\]
With the Cauchy-Schwartz inequality
\[
\left\| 2 \sqrt{\mu} (x_k - x^\star) + v_k \right\|^{2} \leq 2 \left( 4 \mu \left\| x_k - x^\star \right\|^{2} + \left\| v_k \right\|^{2} \right),
\]
and the basic inequality for $f \in \mathcal{S}_{\mu, L}^{1}(\mathbb{R}^n)$
\[
\left\{\begin{aligned} 
         & f(x_{k + 1}) - f(x_{k}) \leq \left\langle \nabla f(x_{k}), x_{k + 1} - x_{k}  \right\rangle + \frac{L}{2} \left\| x_{k + 1} - x_{k} \right\|^{2}  \\
         & f(x^\star) \geq f(x_{k + 1}) + \left\langle \nabla f(x_{k + 1}), x^\star - x_{k + 1} \right\rangle + \frac{\mu}{2} \left\| x^\star - x_{k + 1} \right\|^{2},
        \end{aligned} \right.
\]
we calculate the iterave difference
\begin{align*}
&\lefteqn{\mathcal{E}(k + 1) - \mathcal{E}(k)} \\    & =            \frac{1}{4} \left\langle v_{k + 1} - v_{k}, v_{k + 1} + v_{k} \right\rangle + f(x_{k + 1}) - f(x_{k}) \\
                                                        & \mathrel{ \phantom{=} }   + \frac{1}{4} \left\langle 2 \sqrt{\mu} (x_{k + 1} - x_{k}) + v_{k + 1} - v_{k}, 2 \sqrt{\mu} (x_{k + 1} + x_{k} - 2x^\star) + v_{k + 1} + v_{k} \right\rangle \\
                                                        & \leq        \frac{1}{2} \left\langle v_{k + 1} - v_{k}, v_{k} \right\rangle + \left\langle \nabla f(x_{k}), x_{k + 1} - x_{k} \right\rangle + \frac{L}{2} \left\| x_{k + 1} - x_{k} \right\|^{2}  \\
                                                        & \mathrel{ \phantom{leq} }   + \frac{1}{2} \left\langle 2 \sqrt{\mu} (x_{k + 1} - x_{k}) + v_{k + 1} - v_{k}, 2 \sqrt{\mu} (x_{k} - x^\star) + v_{k} \right\rangle \\
                                                        & \mathrel{ \phantom{\leq} }   + \frac{1}{4} \left\| v_{k + 1} - v_{k} \right\|^{2} + \frac{1}{4} \left\| 2 \sqrt{\mu} (x_{k + 1} - x_{k}) + v_{k + 1} - v_{k} \right\|^{2} \\
                                                        & \leq      - \sqrt{\mu s} \left( \left\| v_{k} \right\|^{2} + \left\langle \nabla f(x_{k}), x_{k} - x^\star \right\rangle \right) \\
                                                        &\mathrel{ \phantom{\leq} }    + \frac{Ls}{2} \left\| v_{k} \right\|^{2} + \frac{s}{4} \left\| 2\sqrt{\mu} v_{k} + \nabla f(x_{k}) \right\|^{2} + \frac{s}{4} \left\| \nabla f(x_{k}) \right\|^{2} \\
                                                        & \leq      - \frac{\sqrt{\mu s}}{2} \left( \left\| v_{k } \right\|^{2} + f(x_{k }) - f(x^\star) + \frac{\mu}{2}\left\| x_{k} - x^\star \right\|^{2}  \right) \\  
                                                        & \mathrel{ \phantom{\leq} }   - \frac{\sqrt{\mu s}}{2} \left( \left\| v_{k} \right\|^{2} +  \frac{1}{L}\left\| \nabla f(x_{k}) \right\|^2 \right) + s \left( 2 \mu + \frac{L}{2} \right) \left\| v_k \right\|^{2} + \frac{3s}{4} \left\| \nabla f(x_{k}) \right\|^{2}. 
\end{align*}
Since $\mu \leq L$, the step size $s \leq \mu/ (25L^{2})$ satisfies it. Hence, the proof is complete after some basic calculations.

\item [(c)]

The Lyapunov function is 
\[
\mathcal{E}(k) = \frac{1}{4} \left\| v_k \right\|^{2} + \frac{1}{4} \left\| 2 \sqrt{\mu} (x_{k + 1} - x^\star) + v_k \right\|^{2} + f(x_k) - f(x^\star).
\]
With the Cauchy-Schwartz inequality
\begin{align*}
\left\| 2 \sqrt{\mu} (x_{k + 1} - x^\star) + v_k \right\|^{2}   &   = \left\| 2 \sqrt{\mu} (x_k - x^\star) + (1 + 2 \sqrt{\mu s}) v_k \right\|^2 \\
                                                                                         &  \leq 2 \left( 4 \mu \left\| x_k - x^\star \right\|^{2} +  (1 + 2 \sqrt{\mu s})^2 \left\| v_k \right\|^{2} \right),
\end{align*}
and the basic inequality for $f \in \mathcal{S}_{\mu, L}^{1}(\mathbb{R}^n)$
\[
\left\{\begin{aligned} 
         & f(x_{k + 1}) - f(x_{k}) \leq \left\langle \nabla f(x_{k + 1}), x_{k + 1} - x_{k}  \right\rangle  - \frac{1}{2L} \left\| \nabla f(x_{k + 1}) - \nabla f(x_{k}) \right\|^{2} \\
         & f(x^\star) \geq f(x_{k + 1}) + \left\langle \nabla f(x_{k + 1}), x^\star - x_{k + 1} \right\rangle + \frac{\mu}{2} \left\| x^\star - x_{k + 1} \right\|^{2},
        \end{aligned} \right.
\]
we calculate the iterave difference
\begin{align*}
& \lefteqn{\mathcal{E}(k + 1) - \mathcal{E}(k)} \\
    & =            \frac{1}{4} \left\langle v_{k + 1} - v_{k}, v_{k + 1} + v_{k} \right\rangle + f(x_{k + 1}) - f(x_{k}) \\
                                                        & \mathrel{ \phantom{=} }   + \frac{1}{4} \left\langle 2 \sqrt{\mu} (x_{k + 2} - x_{k + 1}) + v_{k + 1} - v_{k}, 2 \sqrt{\mu} (x_{k + 2} + x_{k + 1} - 2x^\star) + v_{k + 1} + v_{k} \right\rangle \\
                                                        & \leq        \frac{1}{2} \left\langle v_{k + 1} - v_{k}, v_{k + 1} \right\rangle  - \frac{1}{4} \left\| v_{k + 1} - v_{k} \right\|^{2}  \\
                                                        & \mathrel{ \phantom{\leq} }      + \left\langle \nabla f(x_{k + 1}), x_{k + 1} - x_{k} \right\rangle  - \frac{1}{2L} \left\| \nabla f(x_{k + 1}) - \nabla f(x_{k}) \right\|^{2}  \\
                                                        & \mathrel{ \phantom{\leq} }   + \frac{1}{2} \left\langle - \sqrt{s} \nabla f(x_{k + 1}), 2 \sqrt{\mu} (x_{k + 2} - x^\star) + v_{k + 1} \right\rangle - \frac{1}{4} \left\| \sqrt{s} \nabla f(x_{k + 1}) \right\|^{2} \\
                                                        & \leq      - \sqrt{\mu s} \left( \left\| v_{k + 1} \right\|^{2} + \left\langle \nabla f(x_{k + 1}), x_{k + 1} - x^\star \right\rangle \right) \\
                                                        &\mathrel{ \phantom{\leq} }   - \frac{\sqrt{s}}{2} \left\langle \nabla f(x_{k + 1}), (1 + 2\sqrt{\mu s})v_{k + 1} - v_{k} \right\rangle  \\
                                                        &\mathrel{ \phantom{\leq} }    -  \frac{1}{2L} \left\| \nabla f(x_{k + 1}) - \nabla f(x_{k}) \right\|^{2}  - \frac{1}{4} \left\| v_{k + 1} - v_{k} + \sqrt{s} \nabla f(x_{k + 1}) \right\|^2    \\
                                                        & \leq      -  \sqrt{\mu s} \left[ \left\| v_{k + 1} \right\|^{2} + \frac{1}{4}\left( f(x_{k + 1}) - f(x^\star) \right) + \frac{\mu}{2} \left\| x_{k + 1} - x^\star \right\|^{2}  \right] \\
                                                        & \mathrel{ \phantom{\leq} }    - \frac{1}{4} \left[ 3\sqrt{\mu s} \left( f(x_{k + 1}) - f(x^\star) \right) - 2 s \left\| \nabla f(x_{k + 1}) \right\|^{2} \right].                                          
\end{align*}
Since $\mu \leq L$, the step size $s \leq \mu/ (16L^{2})$ satisfies it. Hence, the proof is complete after some basic calculations.

\end{enumerate}
\end{proof}

\begin{coro}[Discretization of NAG-\texttt{SC} low-resolution ODE]\label{coro:fixed-s-nag-low-three}
For any $f \in \mathcal{S}_{\mu, L}^{1}(\mathbb{R}^n)$, the following conclusions hold:
\begin{enumerate}
\item[(a)]
Taking step size $0 s = \mu/(16L^2)$ , the symplectic Euler scheme satisfies
 \begin{align}
\label{eqn: conver_sc_low_sym_fix}
f(x_k) - f(x^\star) \leq \frac{3L \left\| x_0 - x^\star \right\|^{2} }{2\left(1 + \frac{\mu}{16L} \right)^{k}}.                                       
 \end{align}                  

\item[(b)] Taking step size $s = \mu / (16L^2)$, the explicit Euler scheme satisfies
 \begin{align}
\label{eqn: conver_sc_low_ex_fix}
f(x_k) - f(x^\star) \leq \frac{3L \left\| x_0 - x^\star \right\|^{2} }{2}   \left(1 - \frac{\mu }{40L} \right)^{k}.                                    
 \end{align}

\item[(c)] Taking step size $s = 1/L$, the implicit Euler scheme satisfies
 \begin{align}
\label{eqn: conver_sc_low_im_fix}
f(x_k) - f(x^\star) \leq \frac{3L \left\| x_0 - x^\star \right\|^{2} }{2\left(1 + \frac{1}{4} \sqrt{\frac{\mu}{L}} \right)^{k} }.                                      
 \end{align}   
\end{enumerate}
\end{coro}

\begin{rem}
\label{rem: low_sc}
Compared with Theorem~\ref{thm: low_nag-sc-three} (a) -- (c), just the Euler scheme of the low-resolution ODE~\eqref{eqn: low_hb_NAG-SC}, both the explicit scheme and the symplectic scheme can retain the convergence rate from the continuous version of Theorem~\ref{thm: low_con_strongly}, when the step size $s$ is of the order $O(\mu/L^{2})$. Although the explicit scheme is weaker than the symplectic scheme, it can preserve the rate to the same order as the symplectic scheme. However, if the step size satisfies $s = O(\mu/L^{2})$, the algorithm cannot provide acceleration. There is no limitation on the step size $s$ for the implicit Euler scheme, but in general it is not practical for non-quadratic objective functions.
\end{rem}

\subsection{Low-resolution ODE for convex functions}
\label{subsec: low_c}

In this subsection, we discuss the numerical discretization of~\eqref{eqn: low_NAG-C}. 
We rewrite it in a phase-space representation:
\begin{equation}
\label{eqn: low_NAG-C_phase}
\left\{ \begin{aligned}
          & \dot{X} = V                                             \\
          & \dot{V} = - \frac{3}{t} V - \nabla f(X),  
         \end{aligned} \right. ,
\end{equation}
with $X(0) = x_0$ and $V(0) = 0$.
\begin{thm}
\label{thm: low_con_general}         
Let $f \in \mathcal{F}_{L}^{1}(\mathbb{R}^n)$. The solution $X = X(t)$ to the low-resolution ODE~\eqref{eqn: low_NAG-C} satisfies
\begin{equation}
\label{eqn: conver_c_low1}
\left\{ \begin{aligned}
          & f(X) - f(x^\star) \leq \frac{2 \left\| x_0 - x^\star \right\|^2}{t^2} \\
          & \min_{0 \leq u \leq t} \left\| \nabla f(X(u)) \right\|^{2} \leq \frac{4L \left\| x_0 - x^\star \right\|^{2}}{t^2}.
          \end{aligned}\right. 
 \end{equation}                  
\end{thm}     
Theorem~\ref{thm: low_con_general} is combined with Theorem 3~\cite{su2016differential} and a further analysis about gradient norm minimization  in~\cite{shi2018understanding}. The Lyapunov function is constructed in~\cite{su2016differential} as 
\begin{equation}
\label{eqn: lyapunov_low_c_ode}
\mathcal{E} = t^{2} \left( f(X) - f(x^\star)\right) + \frac{1}{2} \| 2(X - x^\star) + t \dot{X}  \|^{2}. 
\end{equation}

%

\subsubsection{Symplectic Euler scheme}
First, we utilize the symplectic Euler scheme with the initial $x_{0}$ and $v_{0} = 0$, as shown as following:
\begin{equation}
\label{eqn: low_NAG-C_sym}
\left\{ \begin{aligned}
          & x_{k + 1} - x_{k} = \sqrt{s} v_{k}                                                                      \\
          & v_{k + 1} - v_{k} = -  \frac{3}{k + 1} v_{k + 1} - \sqrt{s} \nabla f(x_{k + 1}).  
         \end{aligned} \right. 
\end{equation}

\paragraph{Technical analysis of symplectic scheme~\eqref{eqn: low_NAG-C_sym}}

The Lyapunov function is 
\[
\mathcal{E}(k) = (k + 1)^2 s \left( f(x_{k}) - f(x^\star) \right) + \frac{1}{2} \left\| 2(x_{k + 1} - x^\star) + (k + 1) \sqrt{s} v_{k} \right\|^{2}.
\]
Then we can calculate the iterate difference as
\begin{align*}
      &  \lefteqn{\mathcal{E}(k + 1) - \mathcal{E}(k)} \\
& =     (k + 1)^2 s \left( f(x_{k + 1}) - f(x_{k}) \right)  + (2k + 3) s \left( f(x_{k + 1}) - f(x^\star) \right) \\
& \mathrel{ \phantom{=} }  + \frac{1}{2}\left\langle 2(x_{k + 1} - x_{k}) + (k + 2)\sqrt{s} v_{k + 1} - (k + 1)\sqrt{s} v_{k},\right. \\
& \mathrel{ \phantom{=} } \mathrel{ \phantom{+3\left\langle \right.} }   \left. 2(x_{k + 1} + x_{k} - 2x^\star) + (k + 2) \sqrt{s} v_{k + 1} + (k + 1)\sqrt{s} v_{k}    \right\rangle \\
& =     (k + 1)^{2} s \left( f(x_{k + 1}) - f(x_{k}) \right)  + (2k + 3) s \left( f(x_{k + 1}) - f(x^\star) \right) \\
&\mathrel{ \phantom{=} } - \left\langle (k + 1)s \nabla f(x_{k + 1}),  2(x_{k + 2}  - x^\star) + (k + 2) \sqrt{s} v_{k + 1}  \right\rangle \\
&\mathrel{ \phantom{=} } - \frac{1}{2} (k + 1)^2 s^{2} \left\| \nabla f(x_{k + 1}) \right\|^2. 
\end{align*}
We hope to utilize the basic inequality for $f \in \mathcal{F}_{L}^{1}(\mathbb{R}^n)$ to make the right-hand-side of the equality no more than zero. Based on the following inequalities:
\[
\left\{ \begin{aligned} 
         & f(x_{k + 1}) - f(x_{k})   \leq   \left\langle \nabla f(x_{k + 1}), x_{k + 1} - x_{k}    \right\rangle - \frac{1}{2L} \left\| \nabla f(x_{k + 1}) - \nabla f(x_{k}) \right\|^{2} \\
         & f(x_{k + 1}) - f(x^\star) \leq  \left\langle  \nabla f(x_{k + 1}), x_{k + 1} - x^\star \right\rangle - \frac{1}{2L} \left\| \nabla f(x_{k + 1}) \right\|^{2},
         \end{aligned} \right.
\]
we obtain the following estimate:
\[
\mathcal{E}(k + 1) - \mathcal{E}(k) \leq \frac{1}{2} (k + 1)^{2} s^{2} \left\| \nabla f(x_{k + 1}) \right\|^2 + s \left( f(x_{k + 1}) - f(x^\star) \right) - \frac{(k + 1)s}{L} \left\| \nabla f(x_{k + 1}) \right\|^{2},
\]
from which we cannot guarantee that the right-hand-side of the inequality is nonpositive.


\subsubsection{Explicit Euler scheme}

Now, we turn to the explicit Euler scheme with the initial $x_{0}$ and $v_{0} = 0$, as
\begin{equation}
\label{eqn: low_NAG-C_ex}
\left\{ \begin{aligned}
          & x_{k + 1} - x_{k} = \sqrt{s} v_{k}                                                          \\
          & v_{k + 1} - v_{k} = -  \frac{3}{k} v_{k} - \sqrt{s} \nabla f(x_{k}).  
         \end{aligned} \right. 
\end{equation}

\paragraph{Technical analysis of explicit scheme~\eqref{eqn: low_NAG-C_ex}}

Now, the Lyapunov function is 
\[
\mathcal{E}(k) = (k - 2)(k - 1) s \left( f(x_{k}) - f(x^\star) \right) + \frac{1}{2} \left\| 2(x_{k} - x^\star) + (k - 1) \sqrt{s} v_{k} \right\|^{2}.
\]
Then we can calculate the iterate difference as
\begin{align*}
       &\lefteqn{ \mathcal{E}(k + 1) - \mathcal{E}(k)} \\ 
& =       (k - 1)k s \left( f(x_{k + 1}) - f(x_{k}) \right)  + 2 (k - 1) s \left( f(x_{k}) - f(x^\star) \right) \\
&\mathrel{ \phantom{=} } + \frac{1}{2}\left\langle 2(x_{k + 1} - x_{k}) + k \sqrt{s} v_{k + 1} - (k - 1)\sqrt{s} v_{k},\right. \\
&\mathrel{ \phantom{=} } \mathrel{ \phantom{+2\left\langle \right.} } \left. 2(x_{k + 1} + x_{k} - 2x^\star) + k \sqrt{s} v_{k + 1} + (k - 1)\sqrt{s} v_{k}    \right\rangle \\
&\mathrel{ \phantom{=} }  (k - 1)k s \left( f(x_{k + 1}) - f(x_{k}) \right)  + 2 (k - 1) s \left( f(x_{k}) - f(x^\star) \right) \\
&\mathrel{ \phantom{=} }  + \left\langle - k s \nabla f(x_{k}),  2(x_{k}  - x^\star) + (k - 1) \sqrt{s} v_{k }  \right\rangle  + \frac{1}{2} k^2 s^2 \left\| \nabla f(x_{k}) \right\|^2. 
\end{align*}
\begin{itemize}
\item If we take the following basic inequality for $f \in \mathcal{F}_{L}^{1}(\mathbb{R}^n)$
\[
\left\{ \begin{aligned} 
         & f(x_{k + 1}) - f(x_{k})   \leq   \left\langle \nabla f(x_{k}), x_{k + 1} - x_{k}    \right\rangle + \frac{L}{2} \left\| x_{k + 1} - x_{k} \right\|^{2} \\
         & f(x_{k }) - f(x^\star) \leq  \left\langle  \nabla f(x_{k }), x_{k } - x^\star \right\rangle - \frac{1}{2L} \left\| \nabla f(x_{k }) \right\|^{2},
         \end{aligned} \right.
\]
we obtain the following estimate:
\begin{multline*}
\mathcal{E}(k + 1) - \mathcal{E}(k) \leq \frac{k(k - 1)Ls}{2} \left\| x_{k + 1} - x_{k} \right\|^{2} \\ - 2s \left( f(x_{k}) - f(x^\star) \right) - \frac{ks}{L} \left\| \nabla f(x_{k}) \right\|^{2}+ \frac{k^2 s^2}{2} \left\| \nabla f(x_{k}) \right\|^{2},
\end{multline*}
from which we cannot guarantee that the right-hand-side of the inequality is nonpositive.

\item If we take the following basic inequality for $f \in \mathcal{F}_{L}^{1}(\mathbb{R}^n)$
\[
\left\{ \begin{aligned} 
         & f(x_{k + 1}) - f(x_{k})   \leq   \left\langle \nabla f(x_{k + 1}), x_{k + 1} - x_{k}    \right\rangle - \frac{1}{2L} \left\| \nabla f(x_{k + 1}) - \nabla f(x_{k}) \right\|^{2} \\
         & f(x_{k }) - f(x^\star) \leq  \left\langle  \nabla f(x_{k }), x_{k } - x^\star \right\rangle - \frac{1}{2L} \left\| \nabla f(x_{k }) \right\|^{2},
         \end{aligned} \right.
\]
we obtain the following estimate:
\begin{multline*}
\mathcal{E}(k + 1) - \mathcal{E}(k) \\ \leq  (k - 1) k s \left( \left\langle \nabla f(x_{k + 1}) - \nabla f(x_{k}), x_{k + 1} - x_{k} \right\rangle - \frac{1}{2L} \left\| \nabla f(x_{k + 1}) - \nabla f(x_{k}) \right\|^{2} \right) 
                                                        \\ \mathrel{ \phantom{\leq} } -  2s \left( f(x_{k}) - f(x^\star) \right) - \frac{ks}{L} \left\| \nabla f(x_{k}) \right\|^{2}+ \frac{k^2 s^2}{2} \left\| \nabla f(x_{k}) \right\|^{2},
\end{multline*}
from which we still cannot guarantee that the right-hand-side of the inequality is nonpositive.
\end{itemize}

\subsubsection{Implicit scheme}

Finally, we analyze the implicit Euler scheme with the initial $x_{0}$ and $v_{0} = 0$:
\begin{equation}
\label{eqn: low_NAG-C_im}
\left\{ \begin{aligned}
          & x_{k + 1} - x_{k} = \sqrt{s} v_{k + 1}                                                          \\
          & v_{k + 1} - v_{k} = -  \frac{3}{k + 1} v_{k + 1} - \sqrt{s} \nabla f(x_{k + 1})  
         \end{aligned} \right. 
\end{equation}

\paragraph{Technical analysis of implicit scheme~\eqref{eqn: low_NAG-C_im}}

We construct the Lyapunov function as 
\[
\mathcal{E}(k) = (k + 1)(k + 2) s \left( f(x_{k}) - f(x^\star) \right) + \frac{1}{2} \left\| 2(x_{k} - x^\star) + (k + 1) \sqrt{s} v_{k} \right\|^{2}.
\]
Then we can calculate the iterate difference as
\begin{align*}
      & \lefteqn{ \mathcal{E}(k + 1) - \mathcal{E}(k)} \\ 
       &   =            (k + 1)(k + 2) s \left( f(x_{k + 1}) - f(x_{k}) \right)  + 2(k + 2) s \left( f(x_{k + 1}) - f(x^\star) \right) \\
         &\mathrel{ \phantom{=} }   +            \frac{1}{2}\left\langle 2(x_{k + 1} - x_{k}) + (k + 2)\sqrt{s} v_{k + 1} - (k + 1) \sqrt{s} v_{k},\right. \\
         &\mathrel{ \phantom{=} } \mathrel{ \phantom{+2\left\langle \right.} }\left. 2(x_{k + 1} + x_{k} - 2x^\star) + (k + 2) \sqrt{s} v_{k + 1} + (k + 1) \sqrt{s} v_{k}    \right\rangle \\
       &  =             (k + 1)(k + 2) s \left( f(x_{k + 1}) - f(x_{k}) \right)  + 2(k + 2) s \left( f(x_{k + 1}) - f(x^\star) \right) \\
         &\mathrel{ \phantom{=} }  -            \left\langle (k + 1) s \nabla f(x_{k + 1}),  2(x_{k + 1}  - x^\star) + (k + 2) \sqrt{s} v_{k + 1}  \right\rangle \\
         &\mathrel{ \phantom{=} }  -            \frac{1}{2} (k + 1)^2 s^2 \left\| \nabla f(x_{k + 1}) \right\|^2. 
\end{align*}
Now, we hope to utilize the basic inequality for $f \in \mathcal{F}_{L}^{1}(\mathbb{R}^n)$ to make the right side of equality no more than zero. Based on the following inequalities:
\[
\left\{ \begin{aligned} 
         & f(x_{k + 1}) - f(x_{k})   \leq   \left\langle \nabla f(x_{k + 1}), x_{k + 1} - x_{k}    \right\rangle \\
         & f(x_{k + 1}) - f(x^\star) \leq  \left\langle  \nabla f(x_{k + 1}), x_{k + 1} - x^\star \right\rangle - \frac{1}{2L} \left\| \nabla f(x_{k + 1}) \right\|^{2},
         \end{aligned} \right.
\]
we obtain:
\[
\mathcal{E}(k + 1) - \mathcal{E}(k) \leq 2s \left( f(x_{k + 1}) - f(x^\star) \right) - \frac{(k + 1)s}{L} \left\| \nabla f(x_{k + 1}) \right\|^{2} - \frac{1}{2} (k + 1) s^{2} \left\| \nabla f(x_{k + 1})\right\|^{2}.
\]
Although the negative term includes the multiplier $k$ and $k^2$, 
we cannot guarantee that the right-hand-side of the inequality is nonpositive.


Here, in contrast to the subtle discrete construction in~\cite{su2016differential}, we point out that the standard numerical discretization of low-resolution ODE~\eqref{eqn: low_NAG-C} cannot maintain the convergence rate from the continuous-time ODE, due the presence of numerical error.


\end{document}